%% file: main_stability_Patankar.tex
\documentclass[final,onefignum,onetabnum,letterpaper]{siamonline190516}

\input{main_shared}

\begin{document}

\input{Main_part}

\section*{Acknowledgements}
The author Th. Izgin gratefully acknowledges the financial support by the Deutsche Forschungsgemeinschaft (DFG) through grant ME 1889/10-1. P. \"Offner was supported by the Gutenberg Research College, JGU Mainz. P. \"Offner also wants to thank A. Meister (Kassel) for his invitation to Kassel.

\bibliographystyle{siamplain}
\bibliography{literature}

\end{document}

%% file: main_shared.tex
\usepackage{csquotes}
\usepackage{lipsum}
\usepackage{amsfonts}
\usepackage{epstopdf}
\ifpdf
  \DeclareGraphicsExtensions{.eps,.pdf,.png,.jpg}
\else
  \DeclareGraphicsExtensions{.eps}
\fi

\usepackage{datetime}
\newdateformat{monthyeardate}{%
  \monthname[\THEMONTH] \THEDAY, \THEYEAR}

\usepackage{academicons}
\usepackage{xcolor}
\renewcommand{\orcid}[1]{\href{https://orcid.org/#1}{\textcolor[HTML]{A6CE39}{orcid.org/#1}}}

\usepackage{amsmath}
\allowdisplaybreaks
\usepackage{amssymb}
\usepackage{commath}
\usepackage{mathtools}
\usepackage{bbm}

\usepackage{color}
\usepackage{graphicx}
\usepackage[small]{caption}
\usepackage{subcaption}

\usepackage{relsize}
\usepackage{adjustbox}
\usepackage{algorithm}
\usepackage[noend]{algpseudocode}
\usepackage{booktabs}
\usepackage{tikz}
\usepackage{bm}

\usepackage{verbatim}
\usepackage{ulem}

\usepackage{enumitem}
\setlist[enumerate]{leftmargin=.5in}
\setlist[itemize]{leftmargin=.5in}

\headers{Stability of modified Patankar schemes}{ Th. Izgin  and P. \"Offner}

\title{A study of the local dynamics of modified Patankar DeC and higher order modified Patankar--RK methods\thanks{
\monthyeardate\today 
\corresponding{Philipp \"Offfner} 
\funding{This work was partially supported by the  Deutsche Forschungsgemeinschaft (DFG) through grant ME 1889/10-1 (Izgin) and the Gutenberg Research College, JGU Mainz (\"Offner).}
}}

\author{
Thomas Izgin\thanks{Department of Mathematics, University of Kassel, Germany, (\email{izgin@mathematik.uni-kassel.de}, \orcid{0000-0003-3235-210X}), 
} 
\and 
Philipp \"Offner\thanks{Institute of Mathematics, Johannes Gutenberg University, Mainz, Germany, (\email{poeffner@uni-mainz.de}, \orcid{0000-0002-1367-1917})} 
}

\usepackage{amsopn}

\DeclareMathOperator{\diag}{diag}
\DeclareMathOperator{\Span}{span}

\newtheorem{thm}{Theorem}[section]
\newtheorem{defn}[thm]{Definition}

\newtheorem{prop}[thm]{Proposition}
\newtheorem{rem}[thm]{Remark}
\newtheorem{cor}[thm]{Corollary}

\newcommand{\bp}{\mathbf p}
\newcommand{\bx}{\mathbf x}

\newcommand{\bg}{\mathbf g}
\newcommand{\bA}{\mathbf A}

\newcommand{\bN}{\mathbf N}

\newcommand{\bD}{\mathbf D}

\newcommand{\bI}{\mathbf I}

\newcommand{\bzero}{\mathbf 0}
\newcommand{\bv}{\mathbf v}
\newcommand{\by}{\mathbf y}
\newcommand{\bn}{\mathbf n}

\newcommand{\bbf}{\mathbf f}
\newcommand{\bw}{\mathbf w}

\newcommand{\tm}{\subseteq}

\newcommand{\R}{\mathbb{R}}
\newcommand{\N}{\mathbb{N}}
\newcommand{\C}{\mathbb{C}}

\newcommand{\from}{\colon}
\newcommand{\dt}{\Delta t}
\renewcommand{\prod}{p}
\newcommand{\dest}{d}

\newcommand{\ii}{\mathrm{i}}

\makeatletter
\newcommand{\vast}{\bBigg@{3}}
\newcommand{\Vast}{\bBigg@{4}}
\makeatother

\makeatletter
\newif\ifcenter@asb@\center@asb@false
\def\center@arstrutbox{%
	\setbox\@arstrutbox\hbox{$\vcenter{\box\@arstrutbox}$}%
}
\newcommand*{\CenteredArraystretch}[1]{%
	\ifcenter@asb@\else
	\pretocmd{\@mkpream}{\center@arstrutbox}{}{}%
	\center@asb@true
	\fi
	\renewcommand{\arraystretch}{#1}%
}
\makeatother

\usepackage{tikz}
\usetikzlibrary{angles,quotes}
\usepackage{tkz-euclide}

\definecolor{colorA}{rgb}{0,0.447,0.741}
\definecolor{colorB}{rgb}{0.85,0.325,0.098}
\definecolor{colorE}{rgb}{0.929,0.694,0.125}
\definecolor{colorF}{rgb}{0.494,0.184,0.556}
\definecolor{colorD}{rgb}{0.466,0.674,0.188}
\definecolor{colorC}{rgb}{0.301,0.745,0.933}
\definecolor{colorG}{rgb}{0.635,0.078,0.184}

\makeatletter
\newcommand{\subalign}[1]{%
	\vcenter{%
		\Let@ \restore@math@cr \default@tag
		\baselineskip\fontdimen10 \scriptfont\tw@
		\advance\baselineskip\fontdimen12 \scriptfont\tw@
		\lineskip\thr@@\fontdimen8 \scriptfont\thr@@
		\lineskiplimit\lineskip
		\ialign{\hfil$\m@th\scriptstyle##$&$\m@th\scriptstyle{}##$\hfil\crcr
			#1\crcr
		}%
	}%
}
\makeatother

\makeatletter
\usepackage{environ} 
\newsavebox{\measure@tikzpicture}
\NewEnviron{scaletikzpicturetowidth}[1]{%
	\def\tikz@width{#1}%
	\begin{lrbox}{\measure@tikzpicture}%
		\BODY
	\end{lrbox}%
	\pgfmathparse{#1/\wd\measure@tikzpicture}%
	\BODY
}
\makeatother

%% file: Main_part.tex
\maketitle
\begin{abstract}
Patankar schemes have attracted increasing interest in recent years because they preserve the positivity of the analytical solution of a production-destruction system (PDS) irrespective of the chosen time step size. 
Although they are now of great interest, for a long time it was not clear what stability properties such schemes have. Recently a new stability approach based on Lyapunov stability with an extension of the center manifold theorem has been proposed to study the stability properties of positivity-preserving time integrators.  
In this work, we study the stability properties of the classical modified Patankar--Runge--Kutta schemes (MPRK) and the modified Patankar Deferred Correction (MPDeC) approaches. 
We prove that most of the considered MPRK schemes are stable for any time step size and compute the stability function of MPDeC. We investigate its properties numerically revealing that also most MPDeC are stable irrespective of the chosen time step size. Finally, we verify our theoretical results with numerical simulations. 
\end{abstract}

%
%
\section{Introduction}

The derivation of structure preserving methods is essentially important 
and has attracted much attention recently  in various different ways, cf.  \cite{offner2020arbitrary, kopecz2018unconditionally, ranocha2020general,abgrall2022reinterpretation}. 
One specific family of schemes which ensures the positivity and conservation property of the numerical approximations are modified Patankar (MP) schemes. They are constructed 
for conservative and positive production-destruction systems (PDS) which are used to describe chemical reactions or biological models \cite{meister2003high}, but can also be found as the semi-discretization  of  convection equations (with and without) stiff source terms \cite{ciallella2021arbitrary, huang2019third}. Such a PDS  is defined through the following system of ordinary differential equations (ODEs):
\begin{equation}\label{eq:PDS}
	y_i'=P_i(\by) -D_i(\by), \qquad i=1, \dotsc, N
\end{equation}
with $P_i(\by) = \sum_{j=1}^N p_{ij} (\by)$, $D_i(\by) = \sum_{j=1}^N d_{ij} (\by)$, $p_{ij}(\by)=d_{ji}(\by)\geq 0$ and $\by=(y_1, \dotsc, y_N)^T$, where  $y_i$ denotes the $i$-th concentration, e.\,g.\ the density of a chemical or biological component. Further, the production rate $p_{ij}$ denotes the rate of change from the $j$-th component to the $i$-th  component whereas the destruction $d_{ij}$ tells us the change from the $i$-th component to the $j$-th component. 

The PDS is positive if the componentwise inequality $\by(0)>\bzero$ implies $\by(t)> \bzero $ for all $t>0$.
It is easy to see that due to $p_{ij}(\by)=d_{ji}(\by)$, solutions of \eqref{eq:PDS} fulfill additionally the conservation property $\sum_{i=1}^N y_i(t)=\text{const}$ for all $t\geq 0$. 

To mimic conservation and positivity discretely, modified Patankar schemes have been first developed in \cite{meister2003high} and further extended and applied in \cite{kopecz2018unconditionally, kopecz2019existence, meister2014unconditionally}. The proposed families of second- and third-order accurate modified Patankar–Runge–Kutta (MPRK) methods from \cite{kopecz2018unconditionally, kopecz2019existence, meister2014unconditionally} are based on explicit RK schemes. Huang and Shu have used as their starting point for the  MP developments strong stability preserving (SSP) RK schemes in Shu-Osher form introducing SSPMPRK methods \cite{huang2019positive, huang2019third}. However, all of those approaches have maximum order three, whereas  arbitrary high-order MP schemes have later been proposed in \cite{offner2020arbitrary}. These schemes are    based on the Deferred Correction method \cite{dutt2000spectral}. 
All of those MP schemes are unconditionally positive and conservative for $\by(0)>\bzero$.

Although the benefits of MP schemes have been repeatedly demonstrated numerically since their introduction in 2003, an analytical study of their stability and robustness has been lacking until recently. 
In \cite{torlo2021issues} some issues of MP schemes have been analyzed whereas in  \cite{izgin2021recent,IKM2D, IKMSys}, a suitable stability theory for positivity-preserving time integration schemes has been proposed and applied to a family of second order MPRK schemes. In \cite{huang2022stability}, this theory has been used to investigate the stability properties of the SSPMPRK schemes.
In this paper, we follow the lines of 
\cite{IKM2D, IKMSys} and investigate the stability behavior of several MP schemes. We concentrate on the classical MPRK schemes from \cite{kopecz2018unconditionally, kopecz2019existence} and the MPDeC approach up to order 14 \cite{offner2020arbitrary}.

The rest of the paper is organized as follows.  In the next section, we shortly repeat the notation, the main definitions and results regarding the stability of fixed points from \cite{IKMSys}. Next, in Section \ref{se_MPRK} and \ref{se_dec}, we investigate the stability of MP schemes. First, we focus on a highly efficient MPRK scheme proposed in \cite{torlo2021issues} and explain the general setting. Later, we consider the general families of third order MPRK schemes from \cite{kopecz2018unconditionally,kopecz2019existence}, denoted by MPRK43$(\alpha, \beta)$ and MPRK43$(\gamma)$.
In Section \ref{se_dec}, we extend our stability investigation to the MPDeC schemes from \cite{offner2020arbitrary}. We want to mention here, that the resulting stability functions are included in our repository \cite{ourrepo}.
 In numerical simulations, see Section \ref{se_numerics}, we verify our theoretical results. Finally, we formulate our conclusions and future research interests in section \ref{se_conclusion}. In the Appendix \ref{se_Appendix}, we give some additional material for completeness.
\section{Preliminary Results and Notation}
In the following part we shortly introduce the notation and the considered test equation where we also repeat  the preliminary stability results for positivity-preserving schemes following \cite{izgin2021recent,IKM2D, IKMSys}. 
\subsection{Notation and Test Case}
We consider a stable linear positive and conservative PDS of the form 
\begin{equation}\label{eq_basic}
	\by'(t)=\bA\by(t)
\end{equation}
with $\bA \in \R^{N\times N}$ and the initial condition 
\begin{equation}\label{eq_inital}
	\by(0)=\by^0> \mathbf{0}
\end{equation}
following the notation and investigation in 
\cite{huang2022stability, IKMSys}.
In case that we have exactly $k>0$ linear invariants, there exist 
vectors $\bn_1, \dotsc, \bn_k $ which form a basis of the $\mathrm{ker}(\bA^T)$. As a consequence, they fulfill $\bn_i^T\by(t)= \bn_i^T\by^0$ for all $t\geq0$ and $i=1, \dotsc, k$. Note that for conservative PDS we find $\mathbf{1} \in \mathrm{ker}(\bA^T)$, i.\,e.\ the sum of each column vanishes, that is $\sum_{j=1}^N a_{ji} = 0$ for $i=1,\dotsc, N$. Due to a classical result from dynamical system theory \cite{Luenberger1979}, we know that \eqref{eq_basic} is positive if and only if $\bA$ is a Metzler matrix, i.\,e.\ the off-diagonal elements are nonnegative. Altogether, \eqref{eq_basic} can be rewritten as a positive and conservative PDS with $p_{ij}(\by) = d_{ji}(\by) = a_{ij}y_j$ and $p_{ii}(\by)=d_{ii}(\by)=0$  for $i\neq j$ and $i,j=1,\dotsc,N$. 
Moreover,
\begin{equation}\label{eq:-sum(dij)}
	-\sum_{\substack{j=1}}^Nd_{ij}(\by) = -\sum_{\substack{j=1\\j\neq i}}^Na_{ji}y_i = a_{ii}y_i
\end{equation}
holds true, which was already used in \cite{huang2022stability} to rewrite the SSPMPRK schemes in matrix-vector notation when applied to \eqref{eq_basic}. 

In what follows we are interested in comparing the stability properties of steady states of the initial value problem \eqref{eq_basic}, \eqref{eq_inital} with those of the corresponding fixed points of the method. To that end, we recall the definition of (asymptotically) stable fixed points in the next section and refer to the analogue for steady states to \cite{deuflhard2002,SH98}.

To ensure stable steady states $\by^* \in \ker(\bA)$, the matrix $\bA$ has a spectrum \[\sigma(\bA)\tm \C^-=\{z\in \C\mid \Re(z)\leq 0\}\] and the eigenvalues of $\bA$ with zero real part have to be associated to a Jordan block size of 1, cf. \cite[Theorem 2]{deuflhard2002}. Also, according to the same theorem, no steady state of the test equation \eqref{eq_basic} is asymptotically stable since $\ker(\bA)\neq \{\bzero\}$.

To summarize the assumptions on the linear PDS \eqref{eq_basic}, we introduce the  algebraic multiplicity $\mu_\bA(\lambda)$ of the eigenvalue $\lambda\in \sigma(\bA)$ as well as the corresponding geometric multiplicity $\gamma_\bA(\lambda)$. Altogether we consider in the following systems of the form
\begin{equation}\label{eq:PDS_Sys}
	\by'=\bA\by,\quad \bA\neq\bzero,\quad  \bA-\diag(\bA)\geq \bzero,\quad \mu_\bA(0)=\gamma_\bA(0)=k\geq 1,\quad \bm1\in \ker(\bA^T),\quad \sigma(\bA)\tm \C^-,
\end{equation}
where $\diag(\bA)$ denotes the diagonal of $\bA$. 
\begin{rem}
We want to note, that  $\bm1\in\ker(\bA^T)$ and $\bA\neq \bzero$ imply that the matrix $\bA$ is even a proper Metzler matrix, which means that in addition to being a Metzler matrix, $\bA$ possesses at least one negative diagonal element. As a result of \cite[Theorem 10, Corollary 11]{StabMetz}, the only eigenvalue $\lambda\in \sigma(\bA)\tm\C^-$ with $\Re(\lambda)=0$ is $\lambda=0$. Hence, the only eigenvalue for which we require equality of algebraic and geometric multiplicity in \cite[Theorem 2]{deuflhard2002} is $\lambda=0$.
\end{rem}
As a final notation we want to introduce the diagonal matrix $\diag(\by)$ for a given vector $\by\in \R^N$ by $(\diag(\by))_{ii}=y_i$ for $i=1,\dotsc,N$.

\subsection{Stability of Positive and Conservative Schemes}
Here we recall the main theorem from \cite{IKMSys} that we use for the subsequent stability investigation. The theorem provides sufficient conditions for the stability analysis of non-hyperbolic fixed points, i.\,e.\ fixed points $\by^*$ of a map $\bg\colon D\to D$ with $D\tm \R^N$, where the Jacobian $\bD\bg(\by^*)$ of $\bg$ at $\by^*$ possesses at least one eigenvalue $\lambda$ with $\lvert\lambda\rvert=1.$ We consider the following stability notion of $\by^*$.
\begin{defn}\label{Def_Lyapunov_Diskr}
	Let $\by^*$ be a fixed point of an iteration scheme $\by^{n+1}=\bg(\by^n)$, that is $\by^*=\bg(\by^*)$. 
	\begin{enumerate}
		\item\label{def:stab} $\by^*$ is called \emph{Lyapunov stable} if, for any $\epsilon>0$, there exists a $\delta=\delta(\epsilon)>0$ such that $\norm{\by^0-\by^*}<\delta$ implies $\norm{\by^n- \by^*}<\epsilon$ for all $n\geq 0$.
		\item If in addition to \ref{def:stab}, there exists a constant $c>0$ such that $\Vert \by^0-\by^*\Vert<c$ implies $\Vert \by^n-\by^*\Vert \to 0$ for $n\to \infty$, we call $\by^*$ \emph{asymptotically stable.}
		\item A fixed point that is not Lyapunov stable is said to be \emph{unstable}.
	\end{enumerate}
\end{defn}
In the following we will also just speak of stable fixed points. A common result concerning the stability of fixed points is the following.
\begin{thm}[{\cite[Theorem 1.3.7]{SH98}}]\label{Thm:_Asym_und_Instabil}
	Let  $\by^{n+1}=\bg(\by^n)$ be an iteration scheme with fixed point $\by^*$. Suppose the Jacobian $\bD\bg(\by^*)$ exists and denote its spectral radius by $\rho(\bD\bg(\by^*))$. Then
	\begin{enumerate}
		\item $\by^*$ is asymptotically stable if $\rho(\bD\bg(\by^*))<1$. 
		\item $\by^*$ is unstable if $\rho(\bD\bg(\by^*))>1$.
	\end{enumerate}
\end{thm}
This theorem does not provide us with a criterion for the analysis of fixed points that are stable but not asymptotically stable, which must belong to the class of non-hyperbolic fixed points. Nevertheless, it gives a sufficient condition for detecting unstable fixed points. Still, the case of $\rho(\bD\bg(\by^*))=1$ is not yet covered by Theorem \ref{Thm:_Asym_und_Instabil}, and thus, we use the following theorem from \cite{IKMSys}. For a better reading flow we first introduce some notation.

Since we have assumed that the system $\by'=\bA\by$ has exactly $k\geq 1$ linear invariants, there exist $\bn_1,\dotsc,\bn_k$ forming a basis of $\ker(\bA^T)$. Similar to \cite{IKMSys}, we introduce
\begin{equation*}\label{eq:N}
	\bN=\begin{pmatrix}
		\bn_1^T\\
		\vdots\\
		\bn_k^T
	\end{pmatrix}\in \R^{k\times N} 
\end{equation*}
as well as the set
\begin{equation}
	H=\{\by\in \R^N\mid \bN\by=\bN\by^*\}\label{eq:H},
\end{equation}
and repeat the following theorem on the stability of an iteration in its explicit form, i.\,e.\ $\by^{n+1}=\bg(\by^n)$.
\begin{thm}[\cite{IKMSys}]\label{th_stable}
	Let $\bA\in \R^{N\times N}$ such that $\ker(\bA)=\Span(\bv_1,\dotsc,\bv_k)$ represents a $k$-dimensional subspace of $\R^N$ with $k\geq 1$. Also, let $\by^*\in \ker(\bA)$ be a fixed point of a map $\bg\from D\to D$ where $D\tm \R^N$ contains a neighborhood $\mathcal D$ of $\by^*$. Moreover, let any element of $C=\ker(\bA)\cap \mathcal D$ be a fixed point of $\bg$ and suppose that $\bg\big\vert_{\mathcal{D}}\in \mathcal C^1$ as well as that the first derivatives of $\bg$ are Lipschitz continuous on $\mathcal{D}$. Then $\bD\bg(\by^*)\bv_i=\bv_i$ for $i=1,\dotsc, k$ and the following statements hold.
	\begin{enumerate}
		\item\label{it:Thma} If the remaining $N-k$ eigenvalues of $\bD\bg(\by^*)$ have absolute values smaller than $1$, then $\by^*$ is stable.\label{It:Thm_Stab_a}
		\item\label{it:Thmb} Let $H$ be defined by \eqref{eq:H} and $\bg$ conserve all linear invariants, which means that $\bg(\by)\in H\cap D$ for all $\by\in H\cap D$. If additionally the assumption of \ref{It:Thm_Stab_a} is satisfied, then there exists a $\delta>0$ such that $\by^0\in H\cap D$ and $\norm{\by^0-\by^*}<\delta$ imply $\by^n\to \by^*$ as $n\to \infty$.
	\end{enumerate}
\end{thm} We want to note that according to \cite[Remark 2.4]{huang2022stability}, $\bg\in \mathcal C^2(D)$ is enough to conclude that $\bg\big\vert_\mathcal{D}\in \mathcal C^1$ has Lipschitz continuous derivatives for an appropriate neighborhood $\mathcal D$.

\section{Stability investigation of modified Patankar schemes}\label{se_MPRK}
In \cite{IKM2D, IKMSys},  a novel stability investigation has been developed for conservative and positive time integration schemes. To that end, Theorem \ref{th_stable} was applied to a family of second order MPRK methods and it was shown that all fixed points are unconditional stable, i.\,e.\ stable for all $\dt>0$. Recently, also families of SSPMPRK schemes have been investigated inside this framework \cite{huang2022stability}.  In this section, we  extend those investigations and consider modified Patankar schemes based on several Runge--Kutta methods. Thereby, the schemes are linearly implicit, so that the function $\bg$ satisfying $\by^{n+1}=\bg(\by^n)$ is in all these cases implicitly given by solving linear systems. By the same reasoning as in \cite{huang2022stability}, one can verify that $\bg\in \mathcal C^2$ for all schemes under consideration. The main idea for proving this is to show that the functions $\bm \Phi_i$ defined by the numerical scheme are in $\mathcal{C}^2$ and that $\bg$ is obtained by solving only linear systems with a unique solution. 

Similarly, all schemes preserve positive steady states with the same arguments as in \cite{huang2022stability} or \cite[Prop. 2.3]{torlo2021issues}. Thus, in order to apply Theorem \ref{Thm:_Asym_und_Instabil} and part \ref{it:Thma} of Theorem \ref{th_stable}, we need to compute the eigenvalues of the Jacobian of $\bg$ evaluated at some steady state $\by^*\in \ker(\bA)\cap D^\circ$. However, to use also part \ref{it:Thmb} of Theorem \ref{th_stable}, we need to prove that $\bg$ additionally conserves all linear invariants.

In order to compute $\bD\bg(\by^*)$ we follow the approach as in \cite{IKMSys, huang2022stability} introducing functions $\bm \Phi_i$ and several Jacobians.
The functions $\bm \Phi_i$ arise from rearranging the equations for the $s$ stages of the numerical method leading to
\begin{equation}\label{eq:Schemes:Phi_k}
	\begin{aligned}
		\bzero&=\bm \Phi_j(\by^n,\by^{(1)}(\by^n),\dotsc, \by^{(j)}(\by^n)),\quad j=1,\dotsc,s\\
		\bzero&=\bm \Phi_{n+1}(\by^n,\by^{(1)}(\by^n),\dotsc, \by^{(s)}(\by^n),\bg(\by^n)).
	\end{aligned}
\end{equation}
Note that $\bm \Phi_j(\bx_0,\dotsc,\bx_j)$ is a function of $j+1$ vector-valued variables while $\bm \Phi_{n+1}(\bx_0,\dotsc,\bx_s,\by)$ depends on $s+2$ variables. As mentioned above, in all cases we will find $\bm \Phi_k,\bm\Phi_{n+1}$ are in $\mathcal C^2$, so that we may define 
\begin{equation}\label{eq:jacobians1}
	\begin{aligned}
		\bD_{n}\bm\Phi_j&=\frac{\partial}{\partial \bx_0}\bm\Phi_j, &	\bD_{l}\bm\Phi_j&=\frac{\partial}{\partial \bx_l}\bm\Phi_j,
	\end{aligned}
\end{equation}
for $j,l=1,\dotsc,s$ with $l\leq j$, and
\begin{equation}\label{eq:jacobians2}
	\begin{aligned}
		\bD_{n}\bm\Phi_{n+1}&=\frac{\partial}{\partial \bx_0}\bm\Phi_{n+1}, &\bD_{l}\bm\Phi_{n+1}&=\frac{\partial}{\partial \bx_l}\bm\Phi_{n+1},& &\bD_{n+1}\bm\Phi_{n+1}=\frac{\partial}{\partial \by}\bm\Phi_{n+1}.
	\end{aligned}
\end{equation}
for $l=1,\dotsc, s$. Moreover, the operators $\bD_i^*$ indicate the evaluation of the corresponding Jacobian at $\by^*,\by^{(1)}(\by^*)$ et cetera, e.\,g.\ $\bD^*_n\bm \Phi_2=\bD_n\bm \Phi_2(\by^*,\by^{(1)}(\by^*),\by^{(2)}(\by^*))$. As we interpret $\by^{(k)}=\by^{(k)}(\by^n)$ we also introduce the Jacobian 
\begin{equation*}
	\bD^*\by^{(k)}=\bD\by^{(k)}(\by^*).
\end{equation*}
With that we can derive a formula for computing $\bD\bg(\by^*)$, where $\by^{n+1}=\bg(\by^n)$ is the unique solution to \eqref{eq:Schemes:Phi_k}. 
The chain rule yields
\begin{equation*}
	\begin{aligned}
		\bzero&=\bD^*_{n}\bm\Phi_j+\sum_{l=1}^{j}\bD_{l}^*\bm \Phi_j \bD^*\by^{(l)},\quad j=1,\dotsc, s,\\
		\bzero&=\bD^*_{n}\bm\Phi_{n+1}+\sum_{l=1}^{s}\bD_{l}^*\bm \Phi_{s+1} \bD^*\by^{(l)}+ \bD_{n+1}^*\bm\Phi_{n+1}\bD\bg(\by^*),
	\end{aligned}
\end{equation*}
which can be rewritten to
\begin{equation}\label{eq:FormularJacobian}
	\begin{aligned}
		\bD^*\by^{(j)}&=-\left(\bD^*_{j}\bm\Phi_j\right)^{-1}\left(\bD^*_{n}\bm\Phi_j+\sum_{l=1}^{j-1}\bD_{l}^*\bm \Phi_j \bD^*\by^{(l)}\right),\quad j=1,\dotsc, s,\\
		\bD\bg(\by^*)&=-\left(\bD^*_{n+1}\bm\Phi_{n+1}\right)^{-1}\left(\bD^*_{n}\bm\Phi_{n+1}+\sum_{l=1}^{s}\bD_{l}^*\bm \Phi_{n+1} \bD^*\by^{(l)}\right),
	\end{aligned}
\end{equation}
if all occurring inverses exist. Also, in order to avoid long formulas in the following, we may omit to write the functions $\bm\Phi_i$ together with all their arguments.
\subsection*{Stability of MPRK(3,2)}

In \cite{torlo2021issues} a novel three stage, second order MPRK method has been developed and investigated numerically. The scheme has shown favorable properties and is given by 
\begin{equation}
	\label{eq:MPRK32}
	\begin{aligned}
		y^{(1)}_i &= y^{n}_i
		+ \dt \sum_j \left(
		\prod_{ij}\bigl( \by^n \bigr) \frac{y^{(1)}_j}{y^n_j}
		- \dest_{ij}\bigl( \by^n \bigr) \frac{y^{(1)}_i}{y^n_i}
		\right),
		\\
		y^{(2)}_i &= y^{n}_i
		+ \dt \sum_j \left(
		\frac{\prod_{ij}\bigl( \by^n \bigr) + \prod_{ij}\bigl( \by^{(1)} \bigr)}{4} \frac{y^{(2)}_j}{y^{(1)}_j}
		- \frac{\dest_{ij}\bigl( \by^n \bigr) + \dest_{ij}\bigl( \by^{(1)} \bigr)}{4} \frac{y^{(2)}_i}{y^{(1)}_i}
		\right),
		\\
		y^{n+1} &= y^{n}_i
		+ \dt \sum_j \Biggl(
		\frac{\prod_{ij}\bigl( \by^n \bigr) + \prod_{ij}\bigl(\by^{(1)} \bigr) + 4 \prod_{ij}\bigl( \by^{(2)} \bigr)}{6} \frac{y^{n+1}_j}{y^{(1)}_j}- \frac{\dest_{ij}\bigl( \by^n \bigr) + \dest_{ij}\bigl( \by^{(1)} \bigr) + 4 \dest_{ij}\bigl(\by^{(2)} \bigr)}{6} \frac{y^{n+1}_i}{y^{(1)}_i}
		\Biggr).
	\end{aligned}
\end{equation}
Note that this scheme is steady state preserving as the stages are uniquely determined and $\by^{n+1}=\by^{(2)}=\by^{(1)}=\by^*$ is a valid solution. Hence, writing $\by^{n+1}=\bg(\by^*)$ we obtain $\bg(\by^*)=\by^{(2)}(\by^*)=\by^{(1)}(\by^*)=\by^*$.

Next, we apply MPRK(3,2) to \eqref{eq:PDS_Sys}
which, due to \eqref{eq:-sum(dij)}, yields in matrix vector notation the following identity
\[\bzero=(\bI-\dt\bA)\by^{(1)}-\by^n=\bm \Phi_1\]
and we obtain
\begin{equation}
	\bD_n^*\bm \Phi_1=-\bI\quad \text{ and } \quad  \bD_{1}^*\bm \Phi_1=\bI-\dt\bA.
\end{equation}
The second stage can be written as
the linear system 
\begin{equation}
	\bzero=\by^n+\frac{\dt}{4}\bA\diag(\by^{(2)})(\diag(\by^{(1)}))^{-1}(\by^n+\by^{(1)})-\by^{(2)}=\bm \Phi_2.\label{eq:Phi3yn}
\end{equation}
This equation can be transformed using the fact that diagonal matrices commute and $\diag(\bv)\bw=\diag(\bw)\bv$ for all $\bv,\bw\in \R^N$, yielding
\begin{align}
	\bm \Phi_2&=\by^n+\frac{\dt}{4}\bA\diag(\by^n+\by^{(1)})(\diag(\by^{(1)}))^{-1}\by^{(2)}-\by^{(2)}\label{eq:Phi3y3}\\
	&=\by^n+\frac{\dt}{4}\bA\diag(\by^{(2)})\left(\diag(\by^{n})(\by^{(1)})^{-1}+\bm 1\right)-\by^{(2)},\label{eq:Phi3y2}
\end{align}
where $\bm 1=(1,\dotsc,1)^T\in \R^N$ as well as $(\by^{-1})_i=y_i^{-1}$ for $\by\in \R^N_{>0}$ and $i=1,\dotsc,N$. Hence, we find
\begin{equation}
	\begin{aligned}
		\bD_n^*\bm \Phi_2&\overset{\eqref{eq:Phi3yn}}{=}\bI+\frac{\dt}{4}\bA, \qquad
		\bD_{1}^*\bm \Phi_2&\overset{\eqref{eq:Phi3y2}}{=}-\frac{\dt}{4}\bA,\qquad
		\bD_{2}^*\bm \Phi_2&\overset{\eqref{eq:Phi3y3}}{=}\frac{\dt}{2}\bA-\bI.
	\end{aligned}
\end{equation}
With analogous steps as before, the final step reads
\begin{equation}
	\bzero=\by^n+\frac{\dt}{6}\bA\diag(\by^{n+1})(\diag(\by^{(1)})^{-1}(\by^n+\by^{(1)}+4\by^{(2)})-\by^{n+1}=\bm \Phi_{n+1}. \label{eq:Phin+1yn}
\end{equation}
Using again the above mentioned calculation rules, we get
\begin{align}
	\bm \Phi_{n+1}&=\by^n+\frac{\dt}{6}\bA\diag(\by^n+\by^{(1)}+4\by^{(2)})(\diag(\by^{(1)}))^{-1}\by^{n+1}-\by^{n+1}\label{eq:Phi4yn+1}\\
	&=\by^n+\frac{\dt}{6}\bA\diag(\by^{n+1})\left(\diag(\by^{n}+4\by^{(2)})(\by^{(1)})^{-1}+\bm 1\right)-\by^{n+1}.\label{eq:Phi4y3}
\end{align}
Therefore,
\begin{equation}
	\begin{aligned}
		\bD_n^*\bm \Phi_{n+1}&\overset{\eqref{eq:Phin+1yn}}{=}\bI+\frac{\dt}{6}\bA, &&\quad &
		\bD_{1}^*\bm \Phi_{n+1}&\overset{\eqref{eq:Phi4y3}}{=}-\frac{5}{6}\dt\bA,\\
		\bD_{2}^*\bm \Phi_{n+1}&\overset{\eqref{eq:Phin+1yn}}{=}\frac23\dt\bA, &&\quad &
		\bD_{n+1}^*\bm \Phi_{n+1}&\overset{\eqref{eq:Phi4yn+1}}{=}\dt\bA-\bI.
	\end{aligned}
\end{equation}
Since $\sigma(\bA)\tm \C^-$, we conclude that the inverses of the following matrices exist. We obtain from \eqref{eq:FormularJacobian}:
\begin{equation*}
	\begin{aligned}
		\bD^*\by^{(1)}&=-\left(\bD^*_{1}\bm\Phi_1\right)^{-1}\bD^*_{n}\bm\Phi_1=(\bI-\dt\bA)^{-1},\\
		\bD^*\by^{(2)}&=-\left(\bD^*_{2}\bm\Phi_2\right)^{-1}\left(\bD^*_{n}\bm\Phi_2+\bD_{1}^*\bm \Phi_2 \bD^*\by^{(1)}\right)=-\left(\frac{\dt}{2}\bA-\bI\right)^{-1}\left(\bI+\frac{\dt}{4}\bA-\frac{\dt}{4}\bA(\bI-\dt\bA)^{-1}\right),\\
		\bD\bg(\by^*)
		&=-(\bD_{n+1}^*\bm \Phi_{n+1})^{-1}(\bD_{n}^*\bm \Phi_{n+1}+\bD_{1}^*\bm \Phi_{n+1}\bD^*\by^{(1)}+\bD_{2}^*\bm \Phi_{n+1}\bD^*\by^{(2)})\\
		&=
		(\bI-\dt\bA)^{-1}\Biggl(\bI+\frac{\dt}{6}\bA-\frac{5}{6}\dt\bA(\bI-\dt\bA)^{-1}\\&\hspace{2cm} -\frac23 \dt\bA\left(\frac{\dt}{2}\bA-\bI\right)^{-1}
		\left(\bI+\frac{\dt}{4}\bA-\frac{\dt}{4}\bA(\bI-\dt\bA)^{-1}\right)\Biggr).
	\end{aligned}
\end{equation*}
Hence, eigenvectors of $\bA$ are eigenvectors of the Jacobian $\bD\bg(\by^*)$.  By substituting $z= \lambda \dt$, we obtain the stability function
$R(z)$ for MPRK(3,2) with
\begin{equation}
	\begin{aligned}
		R(z)&=\frac{1}{1-z}\left(1+\frac{z}{6}-\frac{5z}{6(1-z)}-\frac{4z}{3 (z-2)}\left(1+\frac{z}{4}-\frac{z}{4(1-z)}\right)\right)
		=\frac{z^3+18z-12}{6(1-z)^2(z-2)},
	\end{aligned}
\end{equation}
so that \[\sigma(\bD\bg(\by^*))=\{R(\lambda\dt)\mid \lambda\in \sigma(\bA)\}. \]
Before we analyze this stability function, we want to point out that, due to \eqref{eq:Phi4yn+1}, \[\bn^T\bg(\by^n)=\bn^T\by^{n+1}=\bn^T\by^n\] holds true for all $\bn\in \ker(\bA)$ as $\bn^T\bA=\bzero$. Hence, the map $\bg$ generating the iterates of the MPRK(3,2) scheme conserves all linear invariants of the linear test problem. 
As a consequence, all conditions for applying Theorem \ref{Thm:_Asym_und_Instabil} and Theorem \ref{th_stable} are met. It remains to prove that $\abs{R(z)}<1$ holds in order to conclude the stability of the fixed points as well as the local convergence towards the corresponding steady state solution $\by^*$ satisfying $\bn^T\by^0=\bn^T\by^*$ for all $\bn\in \ker(\bA^T)$.
\begin{prop}\label{prop:Stab_MPRK32}
	The stability function $R(z)=\frac{z^3+18z-12}{6(1-z)^2(z-2)}$ satisfies $R(0)=1$ and $\abs{R(z)}< 1$ for $z\in \C^-\setminus \{0\}$.  
\end{prop}
\begin{proof}
	Obviously, it is $R(0)=1$. Then, we investigate $\abs{R(z)}^2$ for $z=\ii y$ and $y\in \R$. Elementary calculations yield for the denominator of $\abs{R(z)}^2$:  
	\begin{equation*}
		\abs{6(1-z)^2(z-2)}^2= 144+324y^2+216y^4+36y^6,
	\end{equation*}
	where the numerator gives 
	\begin{equation*}
		\abs{z^3+18z-12}^2= 144+324y^2-36y^4+y^6.
	\end{equation*}
	Due to the even powers of $y$, it is obviously that for $y\neq 0$, we have $\abs{R(\ii y)}<1$. Since $R(z)$ is a rational function with poles\footnote{Note that the stability function is always well defined on $\C^-$ as otherwise a Jacobian of some $\bm \Phi_i$ must have been singular.} at $z=1,2$, we see that $R(z)$ is holomorphic for all $z \in \C^-$. We can apply the Phgramén-Lindelöf principle\footnote{A special case of the principle is concerned with a holomorphic function $f$ on the interior of $\C^-$, which additionally is continuous on the imaginary axis and growths slower than the function $\exp(C\lvert z\rvert ^\rho)$ for some $C>0$ and $\rho\in[0,1)$. For such a function it is possible to conclude $\lvert f(z)\rvert\leq 1$ for all $z\in\C^-$ if this inequality holds true on the imaginary axis \cite{SS03,T39}.} on the union of the origin and the interior of $\C^-$ and deduce that $\abs{R(z)}\leq 1$ for all $z\in \C^-$.
	With the maximum principle, we know that $\abs{R(z)}=1$ is only possible on the imaginary axis since $R$ is not constant. Therefore, it follows $\abs{R(z)}<1$
	for all $z$ with $\Re(z)<0$.
\end{proof}

From this result, we can follow as a direct consequence of Theorem \ref{th_stable}:
\begin{cor}
	\begin{enumerate}
		\item Let $\by^*$ be a positive steady state of the differential equation \eqref{eq_basic}. Then,  $\by^*$ is a stable fixed point of the MPRK(3,2) for all $\dt>0$.
		\item Let $\by^*$ be the unique steady state of the initial value problem \eqref{eq_basic}, \eqref{eq_inital}. Then,  the scheme MPRK(3,2) converges towards $\by^*$ for all $\dt>0$, if $\by^0$ is close enough to $\by^*$.
	\end{enumerate}
\end{cor}

\begin{rem} \label{rem:Phragmen}
	We want to mention that as long as the stability function $R(z)=\frac{N(z)}{D(z)}$ with polynomials $N,D$ satisfying $\deg(N)\leq\deg(D)$ and $D(z)\neq 0$ for all $z\in \C^-$, we can conclude $\lvert R(z)\rvert<1$ for $\Re(z)<0$ whenever $\lvert R(z)\rvert\leq1$ holds on the imaginary axis with the same reasoning as in the proof of Proposition \ref{prop:Stab_MPRK32}.
	
	Moreover, we point out that the Phragmén-Lindelöf principle can also be applied to different sectors of $\C^-$. 
	Indeed, according to \cite{StabMetz}, all eigenvalues of a proper Metzler matrix $\bA\in \R^{N\times N}$ with $0\in \sigma(\bA)\tm \C^-$ and $N\leq 4$ lie in the sector $S=\left\{z\in \C^-\mid \arg(z)\in(\tfrac34\pi,\tfrac54\pi)\right\}$. Hence, for investigating a scheme when applied to a system with at most four equations, it suffices to investigate the stability function on the boundary of $S$ rather than the imaginary axis. This approach might be useful to give first insights into the stability of a scheme whose analysis for the general case is more involved.
\end{rem}

\subsection*{Stability of MPRK43($\alpha, \beta$)}

In the following part, the third order accurate MPRK43-family from \cite{kopecz2018unconditionally, kopecz2019existence} will be analyzed, which can be written as
{\small{\begin{equation}
	\label{eq:MPRK43-family}
	\begin{aligned}
		y^{(1)}_i &= y^n_i
		+  a_{21} \dt \sum_{j=1}^N \left(
		\prod_{ij}\bigl( \by^n \bigr) \frac{y^{(1)}_j}{y^n_j}
		- \dest_{ij}\bigl( \by^n \bigr) \frac{y^{(1)}_i}{y^n_i}
		\right),
		\\
		y^{(2)}_i &= y^n_i
		+\dt \sum_{j=1}^N
		\Biggl(\left(a_{31} \prod_{ij}\bigl(\by^n\bigr)+ a_{32} \prod_{ij} \bigl(\by^{(1)}\bigr) \right)  \frac{  y_j^{(2)}
		}{\bigl(y_j^{(1)}\bigr)^{\frac1p } \bigl(y_j^n\bigr)^{1-\frac1p} }
		-\left(a_{31} \dest_{ij}\bigl(\by^n\bigr)+ a_{32} \dest_{ij} \bigl(\by^{(1)}\bigr) \right)  \frac{  y_i^{(2)}
		}{\bigl(y_i^{(1)}\bigr)^{\frac1p } \bigl(y_i^n\bigr)^{1-\frac1p} }
		\Biggr),
		\\
		y^{(3)}_i &= y_i^n + \dt \sum_{j=1}^N
		\Biggl(\left( \beta_1 \prod_{ij} \bigl( \by^n\bigr) +\beta_2 \prod_{ij} \bigl(\by^{(1)}\bigr)  \right) \frac{y^{(3)}_j}{\bigl(y_j^{(1)} \bigr)^{\frac1q}
			\bigl(y_j^n\bigr)^{1-\frac1q}}
		-
		\left( \beta_1 \dest_{ij} \bigl( \by^n\bigr) +\beta_2 \dest_{ij} \bigl(\by^{(1)}\bigr)  \right) \frac{y^{(3)}_i}{\bigl(y_i^{(1)} \bigr)^{\frac1q}
			\bigl(y_i^n\bigr)^{1-\frac1q}}\Biggr)
		\\
		y^{n+1}_i &= y^n_i
		+ \dt \sum_{j=1}^N \Biggl(
		\left( b_1 \prod_{ij}\bigl( \by^n \bigr) +b_2\prod_{ij}\bigl( \by^{(1)} \bigr)
		+ b_3 \prod_{ij}\bigl( \by^{(2)} \bigr)
		\right) \frac{y^{n+1}_j}{y^{(3)}_j}
				- \left( b_1 \dest_{ij}\bigl( \by^n \bigr) +b_2\dest_{ij}\bigl( \by^{(1)} \bigr)
		+ b_3 \dest_{ij}\bigl( \by^{(2)} \bigr)
		\right) \frac{y^{n+1}_i}{y^{(3)}_i}
		\Biggr),
	\end{aligned}
\end{equation}}}
where $p=3 a_{21}\left(a_{31}+a_{32} \right)b_3,\; q=a_{21},\;\beta_2=\frac{1}{2a_{21}}$ and $\beta_1= 1-\beta_2$. 

We want to note that in \cite{kopecz2018unconditionally} the stages are labeled $\by^{(1)},\by^{(2)},\by^{(3)}$ and $\bm \sigma$, however, the first stage reads $\by^{(1)}=\by^n$ and can directly be inserted into the further stages. As a result, the first non-trivial stage is the second one, which we denote by $\by^{(1)}$ here. As a result of this shift in superscript numbering, the stage $\bm\sigma$ corresponds to $\by^{(3)}$ in our framework. We performed these changes in order to be consistent with our notation and to reduce the number of functions $\bm \Phi_i$ by one. Also note that by the same reasoning as above, this scheme is also steady state preserving so that $\by^n=\by^*>\bzero$ yields $\by^{(i)}(\by^n)=\by^*$ for $i=1,2,3$ and $\bg(\by^n)=\by^{n+1}(\by^n)=\by^*$.

The Butcher tableau with respect to the two parameters is defined by
\begin{equation}
	\begin{aligned}
		\def\arraystretch{1.2}
		\begin{array}{c|ccc}
			0 &  & & \\
			\alpha & \alpha & & \\
			\beta & \frac{3\alpha\beta (1-\alpha)-\beta^2}{\alpha(2-3\alpha)}& \frac{\beta (\beta-\alpha)}{\alpha(2-3\alpha)}& \\
			\hline
			& 1+\frac{2-3(\alpha+\beta)}{6 \alpha \beta } &\frac{3 \beta-2}{6\alpha (\beta-\alpha)} & \frac{2-3\alpha}{6\beta(\beta-\alpha)}
		\end{array}
	\end{aligned}
\end{equation}
with
\begin{equation}\label{eq:cond:alpha,beta}
	\begin{rcases}
		2/3 \leq \beta \leq 3\alpha(1-\alpha)\\
		3\alpha(1-\alpha)\leq\beta \leq 2/3 \\
		\tfrac{3\alpha-2}{6\alpha-3}\leq \beta \leq 2/3
	\end{rcases}
	\text{ for }
	\begin{cases}
		1/3 \leq \alpha<\frac23,\\
		2/3 \leq \alpha<\alpha_0,\\
		\alpha>\alpha_0,
	\end{cases}
\end{equation}
and $\alpha_0\approx 0.89255.$ \\

We  follow the steps from MPRK32 and obtain 
\begin{equation*}
	\bzero=(\bI-a_{21}\dt\bA)\by^{(1)}-\by^n=\bm \Phi_1.
\end{equation*}
The Jacobians of $\bm\Phi_1$ with respect to $\by^n$ and $\by^{(1)}$ are
\begin{equation}
	\bD_n^*\bm \Phi_1=-\bI\quad \text{ and } \quad  \bD_{1}^*\bm \Phi_1=\bI-a_{21}\dt\bA.
\end{equation}
For the second step, we obtain the linear system
\begin{equation}\label{eq:Phi3yn_2}
	\begin{aligned}
		\bzero=&\dt \bA\diag(\by^{(2)})((\diag(\by^{(1)}))^{-1})^{\frac1p}
		((\diag(\by^{n}))^{-1})^{1-\frac1p}(a_{31}\by^n+a_{32}\by^{(1)})	+\by^n-\by^{(2)}=\bm \Phi_2.
	\end{aligned}
\end{equation}
In order to compute the Jacobians of $\bm \Phi_2$ with respect to $\by^n$ and $\by^{(1)}$, we introduce \[\bbf(\bx,\by(\bx))=\diag(\bx)^k(a_1\bx+a_2\by(\bx))\] for some $k,a_1,a_2\in \R$, where $\by(\bx^*)=\bx^*$. We see that 
\begin{equation}\label{eq:ProductRuleDiag}
	\begin{aligned}
		\left(\frac{\partial}{\partial \bx}\bbf(\bx^*,\by(\bx^*))\right)_{ij}&= \frac{\partial}{\partial x_j}\left((x_i)^k(a_1x_i+a_2y_i(\bx)))\right)\Big|_{\bx=\bx^*}\\&=\delta_{ij}\left(k(x_i^*)^{k-1}(a_1+a_2)x_i^*+(x_i^*)^ka_1\right)\\&=\left(\diag(\bx^*)^{k}\right)_{ij}(k(a_1+a_2)+a_1).
	\end{aligned}
\end{equation}
Hence, the derivatives of $\bm\Phi_2$ are given by the following terms, using $a_{31}+a_{32}=\beta$:
\begin{equation}
	\begin{aligned}
		\bD_n^*\bm \Phi_2&=\bI + \dt \bA \left( \left(\frac1p-1\right)\beta +  a_{31}\right), \\
		\bD_{1}^*\bm \Phi_2&= \dt \bA \left( -\frac{\beta}{p}  +  a_{32}\right), \\
		\bD_{2}^*\bm \Phi_2&=-\bI+\beta\dt\bA.
	\end{aligned}
\end{equation}
Next, $\bm \Phi_3$ satisfies 
\begin{equation}\label{eq:intermedia}
	\begin{aligned}
		\bzero=& \dt\bA\diag(\by^{(3)})((\diag(\by^{(1)}))^{-1})^{\frac1q}
		((\diag(\by^{n}))^{-1})^{1-\frac1q}(\beta_1\by^n+\beta_2\by^{(1)})+\by^n-\by^{(3)}=\bm \Phi_3.
	\end{aligned}
\end{equation}
The derivatives are given by the following terms, using $\beta_1+\beta_2=1$ and \eqref{eq:ProductRuleDiag}:
\begin{equation}
	\begin{aligned}
		\bD_n^*\bm \Phi_3&=\bI + \dt \bA  \left( \frac1q-1 +  \beta_1\right), \\
		\bD_{1}^*\bm \Phi_3&= \dt \bA  \left( -\frac1q +  \beta_2\right), \\
		\bD_{2}^*\bm \Phi_3&= \bzero, \\
		\bD_{3}^*\bm \Phi_3&=-\bI+\dt\bA.
	\end{aligned}
\end{equation}
Finally, the last step can be written as 

\begin{equation}\label{eq:Phi4yn}
	\begin{aligned}
		\bzero=\by^n+\dt\bA\diag(\by^{n+1})((\diag\by^{(3)}))^{-1})
		(b_1\by^n+b_2\by^{(1)}+b_3 \by^{(2)})-\by^{n+1}=\bm \Phi_{n+1}.
	\end{aligned}
\end{equation}
Note that from this equation we can conclude that the generating map $\bg$ conserves all linear invariants as $\bn^T\bA=\bzero$ implies $\bn^T\by^n=\bn^T\by^{n+1}.$
Furthermore, the derivatives are given by the following terms, using \mbox{$b_1+b_2+b_3=1$:}
\begin{equation}
	\begin{aligned}
		\bD_n^*\bm \Phi_{n+1}&=\bI + b_1\dt \bA,  \\
		\bD_{1}^*\bm \Phi_{n+1}&= b_2\dt \bA,  \\
		\bD_{2}^*\bm \Phi_{n+1}&= b_3\dt \bA,  \\
		\bD_{3}^*\bm \Phi_{n+1}&=-\dt\bA,\\
		\bD_{n+1}^*\bm \Phi_{n+1}&=\dt \bA -\bI.
	\end{aligned}
\end{equation}
With $\sigma(\bA)\tm \C^-$ the inverses needed for \eqref{eq:FormularJacobian} exist and we get 
\begin{equation}
	\begin{aligned}
		\bD^*\by^{(1)}&=-\left(\bD^*_{1}\bm\Phi_1\right)^{-1}\bD^*_{n}\bm\Phi_1=(\bI-\alpha\dt\bA)^{-1},\\
		\bD^*\by^{(2)}&=-\left(\bD^*_{2}\bm\Phi_2\right)^{-1}\left(\bD^*_{n}\bm\Phi_2+\bD_{1}^*\bm \Phi_2 \bD^*\by^{(1)}\right)\\
		&=\left(\bI-\beta\dt\bA\right)^{-1}\Biggl(\bI + \dt \bA \left( \left(\frac1p-1\right)\beta +  a_{31}\right)+\dt \bA \left( -\frac{\beta}{p}  +  a_{32}\right)(\bI-\alpha\dt\bA)^{-1}\Biggr),\\
		\bD^*\by^{(3)}&=-\left(\bD^*_{3}\bm\Phi_3\right)^{-1}\left(\bD^*_{n}\bm\Phi_3+\bD_{1}^*\bm \Phi_3 \bD^*\by^{(1)}+\bD_{2}^*\bm \Phi_3 \bD^*\by^{(2)}\right)\\
		&=\left(\bI-\dt\bA\right)^{-1}\Biggl(\bI + \dt \bA  \left(\frac1q-1 +  \beta_1\right)+\dt \bA  \left(-\frac1q +  \beta_2\right)(\bI-\alpha\dt\bA)^{-1}\Biggr),\\
		\bD\bg(\by^*)&=-\left(\bD^*_{n+1}\bm\Phi_{n+1}\right)^{-1}\left(\bD^*_{n}\bm\Phi_{n+1}+\sum_{l=1}^{3}\bD_{l}^*\bm \Phi_{n+1} \bD^*\by^{(l)}\right).
	\end{aligned}
\end{equation}
The stability function can be computed similar to the one of MPRK(3,2) and is given by
\begin{equation*}
	\begin{aligned}
		R(z)=&\frac{1+b_1 z+\frac{b_2 z}{1-\alpha z} +\frac{b_3 z\left(1+z\left( \left(\frac1p-1\right)\beta +  a_{31}\right)+\frac{z\left(-\frac{\beta}{q} +a_{32}\right)}{1-\alpha z}\right)}{1-\beta z}-\frac{z\left(1+z\left(\frac1q -1+\beta_1\right)+\frac{z(-\frac1q+\beta_2)}{1-\alpha z} \right)}{1-z}}{1-z}.
	\end{aligned}
\end{equation*} 
It can be reproduced in our repository \cite{IOEgit}, that this stability function can be rewritten in the form
\begin{equation}\label{eq:Stab_fun_MPRK43(alpha,beta)}
	R(z)=\frac{\sum_{j=0}^4n_jz^j}{\sum_{j=0}^4d_jz^j},
\end{equation}
where
\begin{equation*}
	\begin{aligned}
		n_0&=1, &&&d_0&=1,\\
		n_1&=-(1+\alpha+\beta),&&&d_1&=-(\alpha+\beta+2),
	\end{aligned}
\end{equation*}
\begin{equation*}
	\begin{aligned}
		n_2&=\frac{-6(1+\beta)\alpha^3+(6\beta^2+3)\alpha^2+(6\beta^2-3\beta+1)\alpha+2(\beta-1)}{6\alpha(\beta-\alpha)},\\
		n_3&=\frac{(6\beta-3)\alpha^3+(6\beta^2+1)\alpha^2+(3\beta^2+\beta-1)\alpha-2(\beta-1)}{6\alpha(\beta-\alpha)},
	\end{aligned}
\end{equation*}
\begin{equation*}
	\begin{aligned}
		d_2&=(\beta+2)\alpha+2\beta+1,&&&d_3&=-(\alpha+\beta+2\alpha\beta),\\
		n_4&=-\left(\frac16+\left(\beta-\frac12\right)\alpha\right),&&&d_4&=\alpha\beta.
	\end{aligned}
\end{equation*}
An analytical investigation of the stability function of MPRK43($\alpha,\beta$) for arbitrary finite sized systems $N\times N$ is outside the scope of this paper. This is due to the fact  $R$ is a rational map in $\alpha,\beta$ and $z\in \C^-$ and by expanding $R$ such that the coefficients of $z^j$ are polynomials in $\alpha$ and $\beta$, we see that the maximum total degree appearing in the numerator and denominator is $\deg(\alpha^2\beta^2z^4)=8$. But it can be seen that $\abs{R(z)}<1$ is not true for all $(\alpha,\beta)$ satisfying \eqref{eq:cond:alpha,beta}, see for instance Figure \ref{fig:MPRK43Stab}, where the stability region is highlighted in grey for $(\alpha,\beta)=(\tfrac12,\tfrac23)$. 
\begin{figure}[tb]
	\begin{subfigure}[b]{0.5\textwidth}
		\includegraphics[height=7.8cm]{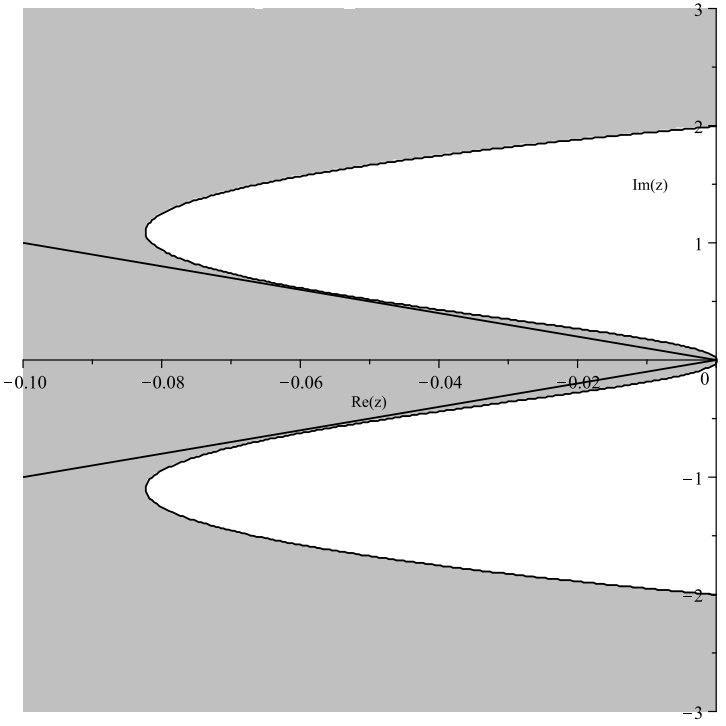}
		\caption{$I_1=[-0.1,0]\times [-3,3]$}
		\label{fig:GL_1}
	\end{subfigure}%
	\begin{subfigure}[b]{0.5\textwidth}
		\includegraphics[height=8cm]{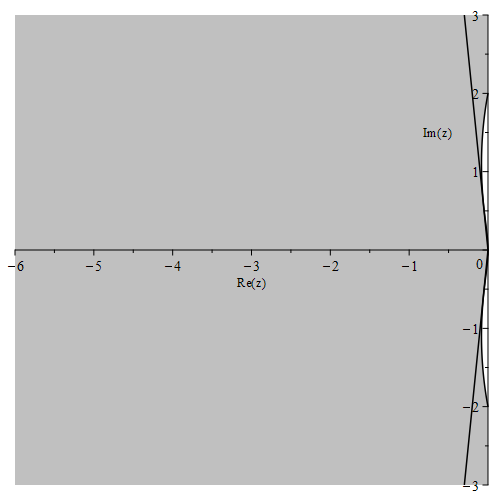} 
		\caption{$I_2=[-6,0]\times [-3,3]$}
		\label{fig:ESP_1}
	\end{subfigure}%
	\caption{The stability domain of MPRK43($\tfrac12,\tfrac23)$ is highlighted in grey and depicted on two rectangles $I_1$ and $I_2$. The two linear segments are given by \mbox{$\Re(z)=-10\abs{\Im(z)}$}. The angle between the two linear segments is approximately $168.58^\circ$.}
	\label{fig:MPRK43Stab}
\end{figure}  

Nevertheless,  if we restrict ourselves to the case outlined in Remark \ref{rem:Phragmen}, i.\,e.\ we investigate $R$ for systems with $N\leq 4$, we can prove that MPKR43($\alpha,\beta$) schemes are stable as we will find that $\abs{R(z)}<1$ on \[S^\circ=\left\{z\in \C^-\setminus\{0\}\mid \arg(z)\in(\tfrac34\pi,\tfrac54\pi)\right\}.\]
Note that for $(\alpha,\beta)=(\tfrac12,\tfrac23)$ the largest cone inside the stability region  contains the sector $S$, see Figure \ref{fig:MPRK43Stab}. Hence, to see the instability properties of MPRK43$(\tfrac12,\tfrac23)$ numerically using a linear positive PDS \eqref{eq:PDS_Sys}, Theorem \cite[Theorem 10]{StabMetz} and a small calculation reveals that we would have to consider a proper Metzler matrix $\bA\in \R^{N\times N}$ with $N\geq 32$. 
\begin{prop}\label{prop:MPRK43(alpha,beta)}
	Let $R$ be defined by \eqref{eq:Stab_fun_MPRK43(alpha,beta)}. Then $R(0)=1$, and, if \begin{equation}\label{eq:assumption.prop.MPRK43(alpha,beta)}
		\begin{aligned} d_3d_4-n_3n_4&\leq 0,&&&
			n_3^2-d_3^2&\leq 0,\\
			n_1n_4-n_2n_3-d_1d_4+d_2d_3&\leq 0&&&\\
			n_2^2-d_2^2-2n_4+2d_4&\leq 0,&&&
			-n_1n_2+d_1d_2+n_3-d_3&\leq 0,
		\end{aligned}
	\end{equation}
	then $\lvert R(z)\rvert<1$ holds true for all $z\in S^\circ=\{z\in \C^-\setminus\{0\}\mid \arg(z)\in(\frac34\pi,\frac54\pi)\}$.
\end{prop}
\begin{proof}
	The following rather technical computations can be reproduced using our repository \cite{IOEgit}.
	First, $n_0=d_0=1$ yields $R(0)=1.$ Next, we make use of Remark \ref{rem:Phragmen} and investigate $R$ on $\partial S$, i.\,e.\ we substitute $z=re^{\ii\varphi}$ with $r>0$ and $\varphi\in \{\frac34\pi,\frac54\pi\}$.
	We obtain\begin{equation*}
		\abs{R(re^{\ii\varphi})}^2=\frac{\left(\sum_{j=0}^4r^jn_j\cos(j\varphi)\right)^2+\left(\sum_{j=0}^4r^jn_j\sin(j\varphi)\right)^2}{\left(\sum_{j=0}^4r^jd_j\cos(j\varphi)\right)^2+\left(\sum_{j=0}^4r^jd_j\sin(j\varphi)\right)^2}.
	\end{equation*}
	Next, note that $\abs{R(re^{\ii\varphi})}^2<1$ is equivalent to 
	\begin{equation} \label{eq:largesum}
		\begin{aligned}
			&\left(\sum_{j=0}^4r^jn_j\cos(j\varphi)\right)^2+\left(\sum_{j=0}^4r^jn_j\sin(j\varphi)\right)^2 -\left(\sum_{j=0}^4r^jd_j\cos(j\varphi)\right)^2-\left(\sum_{j=0}^4r^jd_j\sin(j\varphi)\right)^2\\
			=&(n_4^2 - d_4^2)r^8 + (2n_3n_4 - 2d_3d_4)\cos(\varphi)r^7 + (4(n_2n_4 - d_2d_4)\cos(\varphi)^2 \\&- 2n_2n_4 + n_3^2 + 2d_2d_4 - d_3^2)r^6 + (8(n_1n_4 - d_1d_4)\cos(\varphi)^3 \\&+ (-6n_1n_4 + 2n_2n_3 + 6d_1d_4 - 2d_2d_3)\cos(\varphi))r^5 + (16(n_4 - d_4)\cos(\varphi)^4 \\&+ 4(n_1n_3 - d_1d_3 - 4n_4 + 4d_4)\cos(\varphi)^2 - 2n_1n_3 + n_2^2 + 2d_1d_3 - d_2^2 + 2n_4 - 2d_4)r^4 \\&+ (8(n_3 - d_3)\cos(\varphi)^3 + (2n_1n_2 - 2d_1d_2 - 6n_3 + 6d_3)\cos(\varphi))r^3 \\&+ (4(n_2 - d_2)\cos(\varphi)^2 + n_1^2 - d_1^2 - 2n_2 + 2d_2)r^2 + (2n_1 - 2d_1)\cos(\varphi)r<0.
		\end{aligned}
	\end{equation}
	Dividing by $r>0$ and exploiting $\cos(\tfrac{3}{4}\pi)=\cos(\tfrac54\pi)=-\sqrt{\tfrac12}$ we have to prove that
	\begin{equation*}
		\begin{aligned}
			&(n_4^2 - d_4^2)r^7 + (-n_3n_4 + d_3d_4)\sqrt{2}r^6 + (n_3^2 - d_3^2)r^5 + (n_1n_4 - n_2n_3 - d_1d_4 + d_2d_3)\sqrt2r^4\\
			&+ (n_2^2 - d_2^2 - 2n_4 + 2d_4)r^3 + (-n_1n_2 + d_1d_2 + n_3 - d_3)\sqrt2r^2 \\&+ (n_1^2 - d_1^2)r + (-n_1 + d_1)\sqrt2<0.
		\end{aligned}
	\end{equation*}
	The coefficients of $r^j$ can be written in terms of $\alpha$ and $\beta$ resulting in
	\begin{equation}\label{eq:cvalues}
		\begin{aligned}
			c_7&=n_4^2-d_4^2=\left(\frac13\alpha-\alpha^2\right)\beta+\frac{\alpha^2}{4}-\frac{\alpha}{6}+\frac{1}{36},&&&
			c_6&=d_3d_4-n_3n_4,\\
			c_5&=n_3^2-d_3^2,&&&
			c_4&=n_1n_4-n_2n_3-d_1d_4+d_2d_3,\\
			c_3&=n_2^2-d_2^2-2n_4+2d_4,&&&	
			c_2&=-n_1n_2+d_1d_2+n_3-d_3,\\
			c_1&=n_1^2-d_1^2=-(2\alpha+2\beta+3)<0,&&&
			c_0&=d_1-n_1=-1<0.
		\end{aligned}
	\end{equation}
	Let us consider $c_7$, which is linear in $\beta$. The coefficient of $\beta$ is a parabola in $\alpha$ and vanishes for $\alpha=0$ or $\alpha=\frac13.$ Since for $\alpha=\frac14\in (0,\frac13)$, the coefficient reads $\frac{1}{12}-\frac{1}{16}>0$, we can conclude that the coefficient of $\beta$ is negative for all $\alpha\geq\frac13.$ Now since $\beta\geq\frac14$, the biggest value of $c_7$ can be obtained for $\beta=\frac{1}{4}$, resulting in
	\begin{equation*}
		c_7\left(\alpha,\frac14\right)=-\frac{\alpha}{12}+\frac{1}{36}\leq -\frac{\frac13}{12}+\frac{1}{36}=0.
	\end{equation*}
	Additionally, we assumed that $c_2,\dotsc,c_6\leq0$, so that $\abs{R(z)}<1$ on $S^\circ $ follows from Remark \ref{rem:Phragmen}.
\end{proof}
\begin{rem}
	We want to note that the assumption \eqref{eq:assumption.prop.MPRK43(alpha,beta)} of Proposition \ref{prop:MPRK43(alpha,beta)} can be easily verified for a given $\alpha$ and $\beta$. Indeed, it can be demonstrated that the assumption is always met for $\alpha,\beta$ satisfying \eqref{eq:cond:alpha,beta}, however, an analytical proof is outside the scope of this paper. Instead we present 3D plots of \eqref{eq:assumption.prop.MPRK43(alpha,beta)} in the Appendix \ref{se_Appendix}, see Figures \ref{fig:c2} - \ref{fig:c6}.
\end{rem}
From this result, we can follow as a direct consequence of Theorem \ref{th_stable}:
\begin{cor}
	\begin{enumerate}
		\item Let $\by^*$ be a positive steady state of the differential equation \eqref{eq_basic} with $N\leq 4$. Then, $\by^*$ is a stable fixed point of MPRK($\alpha,\beta$) for all $\dt>0$ and $\alpha,\beta$ satisfying \eqref{eq:cond:alpha,beta} and \eqref{eq:assumption.prop.MPRK43(alpha,beta)}.
		\item If $\by^*$ is  the unique steady state of the initial value problem \eqref{eq_basic} , \eqref{eq_inital} with $N\leq 4$ and if $\by^0$ is close enough to $\by^*$,  then  the MPRK($\alpha,\beta$) iterates converge towards $\by^*$ for  all $\dt>0$  and $\alpha,\beta$ satisfying \eqref{eq:cond:alpha,beta} and \eqref{eq:assumption.prop.MPRK43(alpha,beta)}.
	\end{enumerate}
\end{cor}

\subsection*{Stability of MPRK43($\gamma$)}
Already in \cite{kopecz2019existence} another MPRK43($\gamma$) has been proposed. It is based on
\begin{equation}
	\begin{aligned}
		\def\arraystretch{1.2}
		\begin{array}{c|ccc}
			0 &  & & \\
			\frac23 & \frac23 & & \\
			\frac23 & \frac23-\frac{1}{4\gamma}& \frac{1}{4\gamma}& \\
			\hline
			& \frac14 &\frac34-\gamma & \gamma
		\end{array}
	\end{aligned}
\end{equation}
with $\frac38\leq \gamma \leq \frac34.$ The scheme reads
{\small{\begin{equation}
	\label{eq:MPRK43-family_delta}
	\begin{aligned}
	 y^{(1)}_i &= y^n_i
		+  a_{21} \dt \sum_{j=1}^N \left( \prod_{ij}\bigl( \by^n \bigr)   \frac{y^{(1)}_j}{y^n_j }
		- \dest_{ij}\bigl( \by^n \bigr) \frac{y^{(1)}_i}{y^n_i}
		\right),
		\\
		y^{(2)}_i &= y^n_i
		+\dt \sum_{j=1}^N
		\Biggl(\left(a_{31} \prod_{ij}\bigl(\by^n\bigr)+ a_{32} \prod_{ij} \bigl(\by^{(1)}\bigr) \right) \frac{  y_j^{(2)}
		}{\bigl(y_j^{(1)}\bigr)^{\frac1p } \bigl(y_j^n\bigr)^{1-\frac1p} }-\left(a_{31} \dest_{ij}\bigl(\by^n\bigr)+ a_{32} \dest_{ij} \bigl(\by^{(1)}\bigr) \right) \frac{  y_i^{(2)}
		}{\bigl(y_i^{(1)}\bigr)^{\frac1p } \bigl(y_i^n\bigr)^{1-\frac1p} }\Biggr)
		\\
		y^{(3)}_i &= y_i^n + \dt \sum_{j=1}^N 
		\Biggl(\left(
		\beta_1 \prod_{ij} \bigl( \by^n\bigr) +\beta_2 \prod_{ij} \bigl(\by^{(1)}\bigr) \right)
		\frac{y^{(3)}_j}{\bigl(y_j^{(1)} \bigr)^{\frac1q}
			\bigl(y_j^n\bigr)^{1-\frac1q}}
	-
		\left( \beta_1 \dest_{ij} \bigl( \by^n\bigr) +\beta_2 \dest_{ij} \bigl(\by^{(1)}\bigr)  \right) \frac{y^{(3)}_i}{\bigl(y_i^{(1)} \bigr)^{\frac1q}
			\bigl(y_i^n\bigr)^{1-\frac1q}}\Biggr)
		\\
		y^{n+1}_i &= y^n_i
		+ \dt \sum_{j=1}^N \Biggl(
		\left( b_1 \prod_{ij}\bigl( \by^n \bigr) +b_2\prod_{ij}\bigl( \by^{(1)} \bigr)
		+ b_3 \prod_{ij}\bigl( \by^{(2)} \bigr)
		\right) \frac{y^{n+1}_j}{y^{(3)}_j}
				- \left( b_1 \dest_{ij}\bigl( \by^n \bigr) +b_2\dest_{ij}\bigl( \by^{(1)} \bigr)
		+ b_3 \dest_{ij}\bigl( \by^{(2)} \bigr)
		\right) \frac{y^{n+1}_i}{y^{(3)}_i}
		\Biggr),
	\end{aligned}
\end{equation}}}
where we again renamed the stages to be consistent with our framework. Moreover, the parameters above are given by
\begin{equation*}
	\begin{aligned}
		p&=3a_{21}(a_{31}+a_{32})b_3=\frac43\gamma,&
		q&=a_{21}=\frac23,\\
		\beta_1&=1-\frac{1}{2a_{21}}=\frac14,& \beta_2&=1-\beta_1=\frac34.
	\end{aligned}
\end{equation*}

Following exactly the steps from before, we obtain $\bm\Phi_i=\bzero$ for $i\in\{1,2,3,n+1\}$ where
\begin{equation*}
	\begin{aligned}
		\bm \Phi_1&=\by^n+a_{21}\dt\bA\by^{(1)}-\by^{(1)},\\
		\bm \Phi_2&=\by^n+\dt\bA \diag(\by^{(2)})(\diag(\by^{(1)}))^{-\frac1p}(\diag(\by^{n}))^{\frac1p-1}(a_{31}\by^n+a_{32}\by^{(1)})-\by^{(2)},\\
		\bm \Phi_{3}&=\by^n+\dt\bA \diag(\by^{(3)})(\diag(\by^{(1)}))^{-\frac1q}(\diag(\by^{n}))^{\frac1q-1}(\beta_1\by^n+\beta_2\by^{(1)})-\by^{(3)},\\
		\bm \Phi_{n+1}&=\by^n+\dt\bA \diag(\by^{n+1})(\diag(\by^{(3)}))^{-1}(b_1\by^n+b_{2}\by^{(1)}+b_3\by^{(2)})-\by^{n+1}.
	\end{aligned}
\end{equation*}
Note that $\bn^T\bA=\bzero$ again leads to $\bn^T\by^n=\bn^T\by^{n+1}=\bn^T\bg(\by^n)$ proving that the scheme conserves all linear invariants. 
Furthermore, $\by^n=\by^{(k)}=\by^{n+1}=\by^*$ with $k=1,2,3$ solves $\bm\Phi_1=\bm\Phi_2=\bm\Phi_3=\bm\Phi_{n+1}=\bzero$ so that also this scheme is steady state preserving for positive $\by^*\in \ker(\bA)$. 
Hence, we need to compute $\bD\bg(\by^*)$ to apply Theorem \ref{th_stable} and Theorem \ref{Thm:_Asym_und_Instabil}. We get $
\bD^*_n\bm\Phi_1=\bI$ and $\bD^*_1\bm\Phi_1=a_{21}\dt\bA-\bI,$
which leads to
\[
\bD^*\by^{(1)}=(\bI-a_{21}\dt\bA)^{-1}.\]
Next, using \eqref{eq:ProductRuleDiag} we have
\begin{equation*}
	\begin{aligned}
		\bD^*_n\bm\Phi_2&=\bI+\dt\bA\left(\left(\frac1p-1\right)(a_{31}+a_{32})+a_{31}\right),\\ \bD^*_1\bm\Phi_2&=\dt\bA\left(-\frac1p(a_{31}+a_{32})+a_{32}\right),\\ \bD^*_2\bm\Phi_2&=\dt\bA(a_{31}+a_{32})-\bI,
	\end{aligned} 
\end{equation*}
which determines $\bD^*\by^{(2)}$ but we omit to write it down here for a better reading flow.
Furthermore, 
\begin{equation*}
	\begin{aligned}
		\bD^*_n\bm\Phi_{3}&=\bI+\dt\bA\left(\left(\frac1q-1\right)(\beta_1+\beta_2)+\beta_1\right),\\ \bD^*_1\bm\Phi_{3}&=\dt\bA\left(-\frac1q(\beta_1+\beta_2)+\beta_2\right),\\ \bD^*_{2}\bm\Phi_{3}&=\bzero,\\ \bD^*_{3}\bm\Phi_{3}&=\dt\bA(\beta_1+\beta_2)-\bI.
	\end{aligned} 
\end{equation*}
Finally
\begin{equation*}
	\begin{aligned}
		\bD^*_n\bm\Phi_{n+1}&=\bI+b_1\dt\bA,\quad \bD^*_1\bm\Phi_{n+1}=b_2\dt\bA,\quad \bD^*_2\bm\Phi_{n+1}=b_3\dt\bA\\
		\bD^*_{3}\bm\Phi_{n+1}&=-\dt\bA(b_1+b_2+b_3),\quad \bD^*_{n+1}\bm \Phi_{n+1}=\dt\bA(b_1+b_2+b_3)-\bI.
	\end{aligned}
\end{equation*}
The stability function is thus determined by \eqref{eq:FormularJacobian} reading
\begin{equation}\label{eq:Stab_fun_MPRK43(gamma)}
	\begin{aligned}
		R(z)=&\frac{1}{1-z}\Biggl(1+zb_1+\frac{zb_2}{1-za_{21}}+\frac{zb_3\Biggl(1+z\left(\left(\frac1p-1\right)(a_{31}+a_{32})+a_{31}\right)+\frac{z(-\frac1p(a_{31}+a_{32})+a_{32})}{1-za_{21}}  \Biggr)}{1-z(a_{31}+a_{32})}\\&\hspace{1cm}-\frac{z\left(1+z\left(\left(\frac1q-1\right)(\beta_1+\beta_2)+\beta_1\right)+\frac{z\left(-\frac1q\left(\beta_1+\beta_2\right)+\beta_2\right)}{1-za_{21}} \right)}{1-z(\beta_1+\beta_2)}\Biggr)\\
		=&\frac{1+\frac14z+\frac{z(\frac34-\gamma)}{1-\frac23z}+\frac{z\gamma}{1-\frac23z}\left(1+\frac{z}{4\gamma}+\frac{-\frac{z}{4\gamma}}{1-\frac23z}  \right)-\frac{z\left(1+\frac34z+\frac{-\frac34z}{1-\frac23z} \right)}{1-z}}{1-z}=\frac{-5z^4 + 7z^3 + 23z^2 - 42z + 18}{2(2z-3)^2(z-1)^2}.
	\end{aligned}
\end{equation}
Note that the stability function is independent of the parameter $\gamma$, so that the following investigation is valid for all $\frac38\leq \gamma\leq \frac34.$
\begin{prop}
	Let $R$ be defined by \eqref{eq:Stab_fun_MPRK43(gamma)}. Then $\lvert R(z)\rvert<1$ holds true for all \mbox{$z\in \C^-\setminus\{0\}$} and $R(0)=1$.
\end{prop}
\begin{proof}
	A straightforward calculation yields
	\begin{equation*}
		R(z)=\frac{-\frac{5}{18}z^4 + \frac{7}{18}z^3 + \frac{23}{18}z^2 - \frac{42}{18}z + 1}{(\frac23z-1)^2(z-1)^2}=\frac{\sum_{j=0}^4a_jz^j}{\sum_{j=0}^4b_jz^j},
	\end{equation*}
	where
	\begin{equation*}
		\begin{aligned}
			a_0&=1,&
			a_1&=- \frac{7}{3},&
			a_2&=\frac{23}{18},&
			a_3&=\frac{7}{18},&
			a_4&=-\frac{5}{18},\\
			b_0&=1,&
			b_1&=-\frac{10}{3},&
			b_2&=\frac{37}{9},&
			b_3&=-\frac{20 }{9},&
			b_4&=\frac49.
		\end{aligned}
	\end{equation*}
	Hence $R(0)=1$ and, additionally, we follow the proof of \cite[Prop. 3.6]{huang2022stability}, i.\,e.\ all we have to prove is that
	\begin{equation*}
		\begin{aligned}
			(a_4^2-&b_4^2)y^8+(-2a_2a_4+a_3^2+2b_2b_4-b_3^2)y^6\\
			&+(2a_4-2a_1a_3+a_2^2+2b_1b_3-b_2^2-2b_4)y^4	+(-2a_2+a_1^2-b_1^2+2b_2)y^2<0
		\end{aligned}
	\end{equation*} for all $y\neq 0$. We denote by $c_k$ the above coefficient of $y^k$. A small calculation reveals
	\begin{equation*}
		\begin{aligned}
			c_8&=-\frac{13}{108}<0,&&
			c_6&=-\frac{137}{324}<0,&&
			c_4&=-\frac{1}{12}<0,&&
			c_2&=0,
		\end{aligned}
	\end{equation*}
	which finishes the proof due to the even powers of $y$.
\end{proof}

From this result, we can follow as a direct consequence of Theorem \ref{th_stable}:
\begin{cor}
	\begin{enumerate}
		\item Let $\by^*$ be a positive steady state of the differential equation \eqref{eq_basic}. Then, $\by^*$ is a stable fixed point of MPRK($\gamma$) for all $\dt>0$ and $\frac{3}{8} \leq \gamma \leq \frac{3}{4}$.
		\item If $\by^*$ is  the unique steady state of the initial value problem \eqref{eq_basic}, \eqref{eq_inital},  then,  the scheme MPRK($\gamma$)  converges towards $\by^*$ for  all $\dt>0$  and $\frac{3}{8} \leq \gamma \leq \frac{3}{4}$, if $\by^0$ is close enough to $\by^*$.
	\end{enumerate}
\end{cor}

\section{Stability Investigation of modified Patankar Deferred Correction schemes}\label{se_dec}

Arbitrarily high-order conservative and positive modified Patankar Deferred
Correction schemes (MPDeC) were introduced in \cite{offner2020arbitrary} which is based on the Deferred Correction approach developed in \cite{dutt2000spectral}. 
A time step $[t^n, t^{n+1}]$ is divided into $M$ subintervals, where $t^{n,0}=t^n$
and $t^{n,M}=t^{n+1}$. The idea of the scheme is to mimic the Picard--Lindelöf theorem
as follows. At each subtimestep $t^{n,m}$ and each correction step $k$, an approximation $y^{m,(k)}$ is calculated.
During each of the $K$ correction steps the approximation of the numerical solution is increased by one
order of accuracy.
The modified Patankar trick is introduced inside the basic scheme to guarantee  positivity and conservation of the intermediate approximations.
Using the fact that initial states $y_i^{0,(k)}=y_i^n$
are identical for any correction $k$, the MPDeC correction steps
can be rewritten for $k=1,\dots,K$, $m =1,\dots, M$ and $i=1,\dotsc,N$ as
\begin{equation}
	\label{eq:explicit_dec_correction}
	y_i^{m,(k)}-y_i^n -\sum_{r=0}^M \theta_r^m \dt  \sum_{j=1}^N
	\left( \prod_{ij}(y^{r,(k-1)})
	\frac{y^{m,(k)}_{\gamma(j,i, \theta_r^m)}}{y_{\gamma(j,i, \theta_r^m)}^{m,(k-1)}}
	- \dest_{ij}(y^{r,(k-1)})  \frac{y^{m,(k)}_{\gamma(i,j, \theta_r^m)}}{y_{\gamma(i,j, \theta_r^m)}^{m,(k-1)}} \right)=0,
\end{equation}
where $\theta_r^m$ are the correction weights and the index function $\gamma(j,i,\theta_r^m)$ is defined by
\[\gamma(j,i,\theta_r^m)=\begin{cases}
	j, &\theta_r^m\geq 0,\\
	i,&\theta_r^m< 0.
\end{cases}\]
We also want to note that $\by^{s,(0)}=\by^n$ holds true for all $s=0,\dotsc,M$ by definition. The new numerical solution when applied to \eqref{eq:PDS_Sys} is $\by^{n+1}=\by^{M,(K)}$.

By means of an affine linear transformation of the interval $[t^n,t^{n+1}]$ to $[0,1]$, the $M$ transformed subintervals are determined by $0=t^0< \dots < t^M=1$. Then the correction weights can be computed according to the formula \[\theta_r^m=\int_{0}^{t^m}\varphi_r(t)dt,\] where $\varphi_r$ is the $r$-th Lagrangian basis polynomials defined by the subtimenodes $\lbrace t^m \rbrace_{m=0}^M$. 

Due to the iterative correction process, MPDeC schemes are more complicated than the previous schemes, especially since the index function changes the weights of productive and destructive terms. Since $\gamma$ depends on the sign of $\theta_r^m$, we introduce the nonnegative part $\theta_{m,+}=\min\{0,\theta_r^m\}$ and nonpositive part $\theta_{m,-}=\max\{0,\theta_r^m\}$. It is worth mentioning that 
\begin{equation}
	\begin{aligned}
		\theta_{r,\pm}^m=\frac{\theta_r^m\pm\lvert \theta_r^m\rvert}{2}
	\end{aligned}
\end{equation}
and
\[\theta_r^m=\begin{cases}\theta_{r,-}^m,& \theta_r^m<0,\\
	\theta_{r,+}^m,& \theta_r^m\geq 0\end{cases}\]
as well as $\theta_{r,-}^m+\theta_{r,+}^m=\theta_r^m$.
With that, we split the sum appearing in \eqref{eq:explicit_dec_correction} into two sums containing $\theta_{r,+}^m$ and $\theta_{r,-}^m$, respectively. For the separated sums, we know the value of $\gamma(j,i,\theta_r^m)$ so that we introduce the positive part 
\begin{equation}\label{eq:prk}
	\begin{aligned}
\bp^{r,(k)}(\by^{r,(k-1)},\by^{m,(k-1)},\by^{m,(k)})
=\bA\diag(\by^{m,(k)})\left(\diag(\by^{m,(k-1)})\right)^{-1} \by^{r,(k-1)}
	\end{aligned}
\end{equation}
 as well as  the negative part $\bn^{r,(k)}$ given by
 \begin{equation}\label{eq:nrkpij}
 n_i^{r,(k)}(\by^{r,(k-1)},\by^{m,(k-1)},\by^{m,(k)})=\sum_{j=1}^N \left( \prod_{ij}(\by^{r,(k-1)})
 \frac{y^{m,(k)}_{i}}{y_{i}^{m,(k-1)}}
 - \dest_{ij}(\by^{r,(k-1)})  \frac{y^{m,(k)}_{j}}{y_{j}^{m,(k-1)}} \right)
 \end{equation}
for $i=1,\dotsc, N$, $r=0,\dotsc, M$ and $k=1,\dotsc,K$. Using $\prod_{ij}(\by)=\dest_{ji}(\by)=a_{ij}y_j$ for $i\neq j$ and $p_{ii}(\by)=d_{ii}(\by)=0$ this can be rewritten as
 \begin{equation}\label{eq:nrk}
	n_i^{r,(k)}(\by^{r,(k-1)},\by^{m,(k-1)},\by^{m,(k)})=	\frac{y^{m,(k)}_{i}}{y_{i}^{m,(k-1)}}\sum_{\substack{j=1\\j\neq i}}^N a_{ij}  y_j^{r,(k-1)}
	- y_i^{r,(k-1)}  \sum_{\substack{j=1\\j\neq i}}^N a_{ji}\frac{y^{m,(k)}_{j}}{y_{j}^{m,(k-1)}}.
\end{equation}
Utilizing these vector fields, the iterates from \eqref{eq:explicit_dec_correction} satisfy
\begin{equation}\label{eq:Phis_MPDeC}
	\begin{aligned}
		\bzero=\bm \Phi_k^m(\by^n,\by^{1,(k-1)},\dotsc,\by^{M,(k-1)},\by^{m,(k)})=\by^{m,(k)}-\by^n&-\sum_{r=0}^M \theta_{r,+}^m \dt\bp^{r,(k)}(\by^{r,(k-1)},\by^{m,(k-1)},\by^{m,(k)})\\&-\sum_{r=0}^M \theta_{r,-}^m \dt\bn^{r,(k)}(\by^{r,(k-1)},\by^{m,(k-1)},\by^{m,(k)})
	\end{aligned}
\end{equation}
for $k=1,\dotsc,K$ and $m=1,\dotsc,M$. 
Furthermore, analogously to the Jacobians introduced in \eqref{eq:jacobians1} and \eqref{eq:jacobians2}, we write $\bD^*_x\bm \Phi_k^m$ to represent the Jacobian with respect to the entries of the vector $\by^x$ for some $x$, evaluated at $(\by^*,\dotsc, \by^{m,(k)}(\by^*))$. Finally, we introduce similar notations for the Jacobians of $\bp^{r,(k)}$ and $\bn^{r,(k)}$ with respect to $\by^x$. 

Also note that  MPDeC schemes are steady state preserving as plugging in $\by^{m,(k)}=\by^n=\by^*\in \ker(\bA)$ into \eqref{eq:explicit_dec_correction} yields a true statement. Hence, $\by^{m,(k)}(\by^*)=\by^*$ for all $k=1,\dotsc,K$ and $m=1,\dotsc, M$.

With that, we are able to compute $\bD\bg(\by^*)$ for MPDeC schemes as the next theorem states.
\begin{thm}\label{thm:mpdec}
	Let $\bg:\R^N_{>0}\to\R^N_{>0}$, implicitly given by the solution of \eqref{eq:Phis_MPDeC}, be the generating map of the MPDeC iterates when applied to \eqref{eq:PDS_Sys}. Furthermore, let $\by^*\in \ker(\bA)\cap \R^N_{>0}$ be a steady state of \eqref{eq:PDS_Sys}.
		
	Then, $\bg\in \mathcal C^2$ and the Jacobian of $\bg$ evaluated at $\by^*$ is given by
	\begin{equation}\label{eq:Dg_MPDeC}
		\begin{aligned}
			\bD\bg(\by^*)&=\bD^*_{n}\by^{M,(K)},\\
			\bD^*_{n}\by^{m,(k)}&=-(\bD^*_{m,(k)}\bm \Phi_k^m)^{-1}\left(\bD^*_{n}\bm\Phi_k^m+(1-\delta_{k1})\sum_{r=1}^M\bD^*_{r,(k-1)}\bm \Phi_k^m\bD^*_{n}\by^{r,(k-1)}\right)
		\end{aligned}
	\end{equation}
	for $m=1,\dotsc,M$ and $k=1,\dotsc, K$. Thereby, $\delta_{ij}$ is the Kronecker delta and 
	\begin{equation}\label{eq:D*y0Phikm}
		\begin{aligned}
			\bD^*_{n}\bm \Phi^m_k= \begin{cases}
				-\left(\bI+\dt(\bA +\diag(\by^*)\bA^T(\diag(\by^*))^{-1}) \sum_{r=0}^M \theta_{r,-}^m\right), &k= 1,\\
				-(\bI+\theta_0^m\dt\bA), &k> 1,
			\end{cases}
		\end{aligned}
	\end{equation}
	as well as 
	\begin{equation}\label{eq:D*l,(s)Phikm}
		\begin{aligned}
			\bD^*_{l,(s)}\bm \Phi^m_k=\begin{cases}
				-\theta_l^m\dt\bA, &s=k- 1>0, 0<l\neq m,\\
				\sum_{\substack{r=0}}^M\theta_{r,+}^m\dt\bA-\sum_{\substack{r=0}}^M\theta_{r,-}^m\dt(\diag(\by^*)\bA^T(\diag(\by^*))^{-1})-\theta_{m}^m\dt \bA, &s=k-1>0, l=m,\\
				\bI-\sum_{\substack{r=0}}^M\theta_{r,+}^m\dt\bA+\sum_{\substack{r=0}}^M\theta_{r,-}^m\dt\diag(\by^*)\bA^T(\diag(\by^*))^{-1}, & s=k, l=m.
			\end{cases}
		\end{aligned}
	\end{equation}
\end{thm}
\begin{proof}
	Since the $\theta_r^m$ are fixed for a given scheme, the functions $\bm\Phi_k^m$ are in $\mathcal C^2$ and as a consequence of solving only linear systems, the map $\bg$ is also in $\mathcal C^2$. Furthermore, the formula \eqref{eq:Dg_MPDeC} follows analogously to \eqref{eq:FormularJacobian}, whereby we want to point out that the sum appearing in \eqref{eq:Dg_MPDeC} is multiplied with $0$ for $k=1$ since $\by^{r,(k-1)}=\by^n$ in this case. Hence, we only have to prove the formulas \eqref{eq:D*y0Phikm} and \eqref{eq:D*l,(s)Phikm}. For this, we compute the Jacobians of each addend of the sums in \eqref{eq:Phis_MPDeC} separately by considering \eqref{eq:prk} and \eqref{eq:nrk}.

Let us start proving \eqref{eq:D*y0Phikm}, first considering $k=1$.  From \eqref{eq:prk}  and $\by^{s,(0)}=\by^n$ for all $s=0,\dotsc,M$ it follows that \[\bp^{r,(1)}(\by^n,\by^n,\by^{m,(1)})=\bp^{r,(1)}(\by^n,\by^{m,(1)})=\bA\by^{m,(1)}\] 
	and hence, $\bD^*_{n}\bp^{r,(1)}=\bzero$. 
	Moreover, \eqref{eq:nrk} for $k=1$ yields 
	\[ n_i^{r,(1)}(\by^{n}, \by^{n},\by^{m,(1)})=	n_i^{r,(1)}(\by^{n},\by^{m,(1)})=	\frac{y^{m,(1)}_{i}}{y_{i}^{n}}\sum_{\substack{j=1\\j\neq i}}^N a_{ij}  y_j^{n}
	- y_i^{n}  \sum_{\substack{j=1\\j\neq i}}^N a_{ji}\frac{y^{m,(1)}_{j}}{y_{j}^{n}}.\]
Hence, using $\bm 1\in \ker(\bA^T)$, we obtain
	\begin{equation}\label{eq:2aii}
		\begin{aligned}
\frac{\partial}{\partial y_i^n} n_i^{r,(1)}(\by^*,\by^*)&= -\frac{1}{y_{i}^*}\sum_{\substack{j=1\\j\neq i}}^N a_{ij}  y_j^{*}
- \sum_{\substack{j=1\\j\neq i}}^N a_{ji} =\frac{1}{y_{i}^*}\left(-\sum_{\substack{j=1\\j\neq i}}^N a_{ij}  y_j^{*}
+a_{ii}y^{*}_{i}\right)=\frac{1}{y_{i}^*}\left(-\underbrace{\sum_{\substack{j=1}}^N a_{ij}  y_j^{*}}_{(\bA\by^*)_i=0}
+2a_{ii}y^{*}_{i}\right)=2a_{ii},
		\end{aligned}
	\end{equation}
and for $q\neq i$ we find
	\begin{equation*}
	\begin{aligned}
		\frac{\partial}{\partial y_q^n} n_i^{r,(1)}(\by^*,\by^*)&= a_{iq}
		+ a_{qi}\frac{y^{*}_{i}}{y_{q}^{*}}. 
			\end{aligned}
\end{equation*} 
Altogether, we obtain 
\begin{equation}\label{eq:A+diagAdiag}
	\bD^*_{n}\bn^{r,(1)}=\bA+\begin{pmatrix*}
	a_{11} & a_{21}\tfrac{y_1^*}{y_2^*} &\dots & a_{N1}\tfrac{y_1^*}{y_N^*} \\
	a_{12}\tfrac{y_2^*}{y_1^*} & \ddots\hphantom{\tfrac{y_2^*}{y_1^*}}& & \vdots \\
	\vdots & & \ddots & \vdots\\ 
	a_{1N}\tfrac{y_N^*}{y_1^*}& \dots &\dots & a_{NN}
	\end{pmatrix*}=\bA +\diag(\by^*)\bA^T(\diag(\by^*))^{-1},
\end{equation}
and thus,
\begin{equation*}
	\bD^*_{n}\bm \Phi^m_1=-\left(\bI+(\bA +\diag(\by^*)\bA^T(\diag(\by^*))^{-1}) \Delta t\sum_{r=0}^M \theta_{r,-}^m\right).
\end{equation*}
Next, for $k> 1$ it follows from \eqref{eq:prk} that 
\[\bD^*_{n}\bp^{r,(k)}=\delta_{r0} \bA. \]
Similarly,  $\bD^*_{n}\bn^{r,(k)}=\bzero$ if $r\neq 0$. Furthermore,
\[	n_i^{0,(k)}(\by^{n},\by^{m,(k-1)},\by^{m,(k)})=	\frac{y^{m,(k)}_{i}}{y_{i}^{m,(k-1)}}\sum_{\substack{j=1\\j\neq i}}^N a_{ij}  y_j^{n}
- y_i^{n}  \sum_{\substack{j=1\\j\neq i}}^N a_{ji}\frac{y^{m,(k)}_{j}}{y_{j}^{m,(k-1)}}\]
yields
\begin{equation*}
	\begin{aligned}
		\frac{\partial}{\partial y_i^n} n_i^{0,(k)}(\by^*,\by^*,\by^*)&=- \sum_{\substack{j=1\\j\neq i}}^N a_{ji}=a_{ii}\\
		\frac{\partial}{\partial y_q^n} n_i^{0,(k)}(\by^*,\by^*,\by^*)&=a_{iq}, \quad i\neq q,
	\end{aligned}
\end{equation*}
so that $\bD^*_{n}\bn^{r,(k)}=\delta_{r0}\bA$. This results in
\begin{equation*}
	\bD^*_{n}\bm \Phi^m_k=-\left(\bI+ (\theta_{0,-}^m+\theta_{0,+}^m)\Delta t\bA\right)=-\left(\bI+ \theta_{0}^m\Delta t\bA\right),
\end{equation*}
proving  \eqref{eq:D*y0Phikm}. 

To derive \eqref{eq:D*l,(s)Phikm} consider first the case $s=k-1>0$ and $0<l\neq m$. From \eqref{eq:prk} it follows immediately that
\begin{equation*}
\bD^*_{l,(k-1)}\bp^{r,(k)}=\delta_{rl}\bA.
\end{equation*}
Moreover, \eqref{eq:nrk} yields
\begin{equation*}
	\begin{aligned}
		\frac{\partial}{\partial y_i^{l,(k-1)}} n_i^{r,(k)}(\by^*,\by^*,\by^*)&=-\delta_{rl} \sum_{\substack{j=1\\j\neq i}}^N a_{ji}=\delta_{rl}a_{ii}\\
		\frac{\partial}{\partial y_q^{l,(k-1)}} n_i^{r,(k)}(\by^*,\by^*,\by^*)&=\delta_{rl}a_{iq}, \quad i\neq q,
	\end{aligned}
\end{equation*} 
which means that $\bD^*_{l,(k-1)}\bn^{r,(k)}=\delta_{rl}\bA$ for $l\neq m$. In total \eqref{eq:Phis_MPDeC} gives us
\begin{equation*}
\bD^*_{l,(k-1)}\bm \Phi_k^m=-\dt (\theta_{l,+}^m+\theta_{l,-}^m)\bA=-\dt\theta_l^m\bA.
\end{equation*}

Next, we investigate the case of $s=k-1>0$ and $l=m$. Using $\diag(\bv)\bw=\diag(\bw)\bv$  and \eqref{eq:prk}, we obtain
\begin{equation*}
	\begin{aligned}
		\bD^*_{m,(k-1)}\bp^{r,(k)}=\bD^*_{m,(k-1)}\left(\bA\diag(\by^{m,(k)})\diag\left(\by^{m,(k-1)}\right)^{-1} \by^{r,(k-1)}\right)=-(1-\delta_{rm})\bA.
	\end{aligned}
\end{equation*}
Furthermore, recalling \eqref{eq:nrk}, i.\,e.\
 \begin{equation*}
	n_i^{r,(k)}(\by^{r,(k-1)},\by^{m,(k-1)},\by^{m,(k)})=	\frac{y^{m,(k)}_{i}}{y_{i}^{m,(k-1)}}\sum_{\substack{j=1\\j\neq i}}^N a_{ij}  y_j^{r,(k-1)}
	- y_i^{r,(k-1)}  \sum_{\substack{j=1\\j\neq i}}^N a_{ji}\frac{y^{m,(k)}_{j}}{y_{j}^{m,(k-1)}}.
\end{equation*}
we also distinguish between $r=m$ and $r\neq m$. In the first case we see $n_i^{m,(k)}=n_i^{m,(k)}(\by^{m,(k-1)},\by^{m,(k)})$ and 
\begin{equation*}
	\begin{aligned}
		\frac{\partial}{\partial y_i^{m,(k-1)}} n_i^{m,(k)}(\by^*,\by^*)&=-\frac{1}{y_i^*} \sum_{\substack{j=1\\j\neq i}}^N a_{ij}y_j^*-\sum_{\substack{j=1\\j\neq i}}^N a_{ji}\overset{\eqref{eq:2aii}}{=}2a_{ii}\\
		\frac{\partial}{\partial y_q^{m,(k-1)}} n_i^{m,(k)}(\by^*,\by^*)&=a_{iq}+a_{qi}\frac{y_i^*}{y_q^*}\overset{\eqref{eq:A+diagAdiag}}{=}(\bA+\diag(\by^*)\bA^T(\diag(\by^*))^{-1})_{iq}, \quad i\neq q,
	\end{aligned}
\end{equation*}
which means that $	\bD^*_{m,(k-1)}\bn^{m,(k)}=\bA +\diag(\by^*)\bA^T(\diag(\by^*))^{-1}$. Turning to the case $r\neq m$, we find
\begin{equation*}
	\begin{aligned}
		\frac{\partial}{\partial y_i^{m,(k-1)}} n_i^{r,(k)}(\by^*,\by^*,\by^*)&=-\frac{1}{y_i^*} \sum_{\substack{j=1\\j\neq i}}^N a_{ij}y_j^*\overset{\eqref{eq:2aii}}{=}a_{ii},\\
		\frac{\partial}{\partial y_q^{m,(k-1)}} n_i^{r,(k)}(\by^*,\by^*,\by^*)&=a_{qi}\frac{y_i^*}{y_q^*}\overset{\eqref{eq:A+diagAdiag}}{=}(\diag(\by^*)\bA^T(\diag(\by^*))^{-1})_{iq}, \quad i\neq q,
	\end{aligned}
\end{equation*}
resulting in  $	\bD^*_{m,(k-1)}\bn^{r,(k)}=\diag(\by^*)\bA^T(\diag(\by^*))^{-1}$ for $r\neq m$. Altogether, we thus end up with
\begin{equation*}
	\bD^*_{m,(k-1)}\bm \Phi^m_k=\sum_{\substack{r=0}}^M\theta_{r,+}^m\dt\bA-\sum_{\substack{r=0}}^M\theta_{r,-}^m\dt(\diag(\by^*)\bA^T(\diag(\by^*))^{-1})-\theta_{m}^m\dt \bA.
\end{equation*}

Finally, we have to consider the case $s=k$ and $l=m$, i.\,e.\ we have to compute $\bD^*_{m,(k)}\bm \Phi^m_k$.  Using $\diag(\bv)\bw=\diag(\bw)\bv$  and \eqref{eq:prk} once again we see that 
\[\bD^*_{m,(k)}\bp^{r,(k)}=\bA. \]
Furthermore, we obtain
\begin{equation*}
	\begin{aligned}
		\frac{\partial}{\partial y_i^{m,(k)}} n_i^{r,(k)}(\by^*,\by^*,\by^*)&=\frac{1}{y_i^*} \sum_{\substack{j=1\\j\neq i}}^N a_{ij}y_j^*\overset{\eqref{eq:2aii}}{=}-a_{ii},\\
		\frac{\partial}{\partial y_q^{m,(k)}} n_i^{r,(k)}(\by^*,\by^*,\by^*)&=-a_{qi}\frac{y_i^*}{y_q^*}\overset{\eqref{eq:A+diagAdiag}}{=}-(\diag(\by^*)\bA^T(\diag(\by^*))^{-1})_{iq}, \quad i\neq q,
	\end{aligned}
\end{equation*}
resulting in
\begin{equation*}
	\bD^*_{m,(k)}\bm \Phi^m_k=\bI-\sum_{\substack{r=0}}^M\theta_{r,+}^m\dt\bA+\sum_{\substack{r=0}}^M\theta_{r,-}^m\dt\diag(\by^*)\bA^T(\diag(\by^*))^{-1}
\end{equation*}
With this, we have finally proven Theorem \ref{thm:mpdec}.

\end{proof}
The distribution and number  $M$ of the subtimesteps as well as the number of iterations $K$ determines the order of accuracy of the scheme. In the following, we will use equispaced and Gauss--Lobatto points. To reach order $p$, we use $K=p$ corrections and $M=\max\{K-1,1\}$ subintervals for equispaced point distributions. Focusing on Gauss-Lobatto, a higher-order quadrature rule is applied\footnote{The $L^2$ operator inside the DeC framework is based on a collocation method with Lobatto nodes (also known as the RK Lobatto III A method). }. Here, we  use M= ceil(K/2) subintervals and K=p corrections. 
We will denote the $p$-th order MPDeC method by MPDeC$(p)$.
Note that MPDeC(1) is equivalent to the modified Patankar--Euler scheme  and MPDeC(2) is equivalent to MPRK(2,2,1) investigated in \cite{IKM2D}.
Due to $\by^{n+1}=\by^{M,(K)}$, MPDeC conserves all linear invariants, if $\theta_r^M\geq 0$ for all $r=0,\dotsc,M$ since in this case the index function yields $\gamma(j,i,\theta_r^M)=j$ and \eqref{eq:explicit_dec_correction} can be written as 
	\[\by^{n+1}-\by^n-\sum_{r=0}^M\theta_r^M\dt \bA \diag(\by^{n+1})(\diag(\by^{M,(K-1)})^{-1}\by^{r,(K-1)}=\bzero,\]
	which means that $\bn^T\by^{n+1}=\bn^T\by^n$ for all $\bn\in \ker(\bA^T)$. Indeed, for equispaced nodes, $\theta_r^M$ with $r=0,\dotsc,M$ are the weights of the closed Newton--Cotes formulas for integrals over $I=[0,1]$. Hence, a negative $\theta_r^M$ occurs for the first time at $M=7$, i.\,e.with MPDeC$(8)$. In this case, we also have to consider $\bn^{r,(K)}(\by^{r,(K-1)},\by^{M,(K-1)},\by^{n+1})$ given in \eqref{eq:nrkpij}, resulting in
	\begin{equation*}
		\begin{aligned}
		\bn^T\bn^{r,(K)}(\by^{r,(K-1)},\by^{M,(K-1)},\by^{n+1})&=\sum_{i,j=1}^Nn_i\frac{y_i^{n+1}}{y_i^{M,(K-1)}}p_{ij}(\by^{r,(K-1)})- \sum_{i,j=1}^Nn_i\frac{y_j^{n+1}}{y_j^{M,(K-1)}}d_{ij}(\by^{r,(K-1)})\\
		&=\sum_{i,j=1}^Nn_i\frac{y_i^{n+1}}{y_i^{M,(K-1)}}p_{ij}(\by^{r,(K-1)})- \sum_{i,j=1}^Nn_j\frac{y_i^{n+1}}{y_i^{M,(K-1)}}p_{ij}(\by^{r,(K-1)}),
		\end{aligned}
	\end{equation*}
where we switched indices and used $d_{ij}=p_{ji}$ for the last equality. We observe that $\bn^T\bn^{r,(k)}$ does not need to vanish for $\bn\notin\Span(\bm 1)$, so that the preservation of all linear invariants can not be guaranteed anymore for arbitrary systems \eqref{eq:PDS_Sys} and MPDeC($p$) with equispaced nodes and $p\geq 8$.

	Moreover, in the case of Gauss--Lobatto nodes, the values $2\theta_r^M$ for $r=0,\dotsc,M$ equal the weights of the corresponding Gauss-Lobatto quadrature, which are always positive. As a result, MPDeC with Gauss--Lobatto nodes conserve all linear invariants when applied to \eqref{eq:PDS_Sys}.
\begin{rem}\label{rem:Mpdec}
From Theorem \ref{thm:mpdec}, we see that the Jacobian in general depends on $\by^*$, if there exist negative correction weights $\theta_r^m$. For equispaced or Gauss--Lobatto points this is already the case for $K>2$. Hence, to study the stability of MPDeC schemes applied to general linear systems, one needs to locate the eigenvalues of the Jacobian, which possibly depend on $\by^*$ themselves. Such an analysis is outside the scope of this work, which is why we will focus on the following class of problems.

If $\bA$ is normal, then  $\bA$ and $\bA^T$ share the same eigenvectors and the corresponding eigenvalues are the complex conjugate of each other. Since $\bm 1\in \ker(\bA^T)$ this means that even $\bm1 \in \ker(\bA)$. Hence, we may discuss the stability of $\by^*=\bm 1$. Then, we find
\[\sigma\left(r_1(\bA)+r_2\left(\diag(\by^*)\bA^T(\diag(\by^*))^{-1}\right)\right)=\{r_1(\lambda)+r_2(\bar\lambda)\mid \lambda\in \sigma(\bA)\} \]
for any rational maps $r_1,r_2$, which means that the spectrum of the Jacobian of the map $\bg$ generating  the MPDeC iterates can be written only in terms of the eigenvalues of $\bA$. Using \eqref{eq:Dg_MPDeC}, \eqref{eq:D*y0Phikm} and \eqref{eq:D*l,(s)Phikm}, the stability function $R_p$ of MPDeC($p$) for normal matrices $\bA$ and $\by^*=\bm 1$ can be computed by
\begin{equation}\label{eq:Rmpdec}
	\begin{aligned}
		R^{m,(1)}(z)&=\frac{1+(z+\bar{z})\sum_{r=0}^M\theta_{r,-}^m}{1-\left(z \sum_{r=0}^M \theta_{r,+}^m-\bar{z}\sum_{r=0}^M \theta_{r,-}^m\right)},\\
		R^{m,(\hat k)}(z)&=\frac{1+\theta_0^mz+z\displaystyle\sum_{\substack{r=1\\r\neq m}}^M\theta_r^mR^{r,(\hat k-1)}(z)-\left( z \sum_{r=0}^M \theta_{r,+}^m-\bar{z} \sum_{r=0}^M \theta_{r,-}^m-z\theta_m^m\right)R^{m,(\hat k-1)}(z)}{1-\left(z \sum_{r=0}^M \theta_{r,+}^m-\bar{z}\sum_{r=0}^M \theta_{r,-}^m\right)},\\
		R_p(z)&=R^{M,(K)}(z),
	\end{aligned}
\end{equation}
for $\hat k=2,\dotsc,K$ and $m=1,\dotsc,M$. Note that if $\bA$ is symmetric it is also normal and we obtain $\sigma(\bA)\tm \R$, so that one can further simplify \eqref{eq:Rmpdec} to
\begin{equation}\label{eq:Rmpdecreduced}
	\begin{aligned}
		R^{m,(1)}(z)&=\frac{1+2z\sum_{r=0}^M\theta_{r,-}^m}{1-z \sum_{r=0}^M \abs{\theta_{r}^m}},\\
		R^{m,(\hat k)}(z)&=\frac{1+\theta_0^mz+z\displaystyle\sum_{\substack{r=1\\r\neq m}}^M\theta_r^mR^{r,(\hat k-1)}(z)-z\left(\sum_{r=0}^M \abs{\theta_{r}^m}-\theta_m^m\right)R^{m,(\hat k-1)}(z)}{1-z \sum_{r=0}^M \abs{\theta_{r}^m}},\\
		R_p(z)&=R^{M,(K)}(z),
	\end{aligned}
\end{equation}
using $\theta^m_{r,+}+\theta^m_{r,-}=\theta_r^m$. 
It is also worth mentioning that for the system matrix  
\begin{equation}\label{eq:A2x2}
\bA=\begin{pmatrix*}[r]
	-a & b\\a&-b
\end{pmatrix*}
\end{equation} with $a,b>0$, used in \cite{IKM2D,torlo2021issues}, we find that $\ker(\bA)=\Span(\by^*)$ with $\tfrac{y_1^*}{y_2^*}=\tfrac{b}{a}$, and thus
\[\diag(\by^*)\bA^T(\diag(\by^*))^{-1}=\begin{pmatrix*}[r]
	-a & a\frac{b}{a}\\
	 b\frac{a}{b}&-b
\end{pmatrix*}=\bA,\]
so that the stability function $R_p$ in this case is also given by \eqref{eq:Rmpdecreduced}. 
\end{rem}
Deferred Correction schemes are described by an iterative process which can be compared with classical RK schemes with more stages \cite{abgrall2022relaxation}. As MPDeC and DeC share the same amount of stages, we thus know that MPDeC(3) contains 5 stages using equispaced point distributions. Furthermore, MPDeC(4) has already 10 stages inside for equispaced points distributions, resulting in rational function with polynomial degree 10 in the numerator and denominator.
Using Gauss-Lobatto nodes decreases the number of stages. For an MPDeC(4) we would end up with seven stages. 
By going even to arbitrary high-order the stability analysis would even become more complicated. Therefore, 
we avoid  a theoretical analysis of the stability functions and investigate them numerically. 

To give a first insight in the stability properties of MPDeC, we analyze the reduced stability function \eqref{eq:Rmpdecreduced}. In both cases described in Remark \ref{rem:Mpdec}, the eigenvalues of $\bA$ leading to \eqref{eq:Rmpdecreduced} are real, so that it is no loss of generality to restrict ourselves to the $2\times 2$ case with $\bA$ given in \eqref{eq:A2x2}. 

In Figure \ref{fig:stability}, we present the absolute value of the stability function over real $z$. 
To obtain a stable scheme, the absolute value of $R(z)$ has to be always smaller than one (the black line). In \ref{fig:GL}, we investigate MPDeC from order 4 to 14 using Gauss--Lobatto points. As can be recognized all MPDeC methods are stable using Gauss--Lobatto points. In \ref{fig:ESP}, the stability functions of MPDeC schemes from 4th to 14th order are depicted for equispaced nodes. Here, we recognize that MPDeC(12) and MPDeC(14) are unstable but MPDeC(13) is stable.  However, this is not surprising since already in classical DeC using equidistant points has been problematic for high-order methods, cf.  \cite{dutt2000spectral,han2021dec,offner2020arbitrary, torlo2021issues} and references therein.   The reason for this is related with classical interpolation theory where it is known that equidistant points should be avoided. However, we would like to point out that our investigation supports as well the numerical investigation in \cite{torlo2021issues} where problems in MPDeC-equidistant have been also recognized. 

\begin{figure}[tb]
	\centering 
	\begin{subfigure}[b]{0.49\textwidth}
		\includegraphics[width=\textwidth]{%
			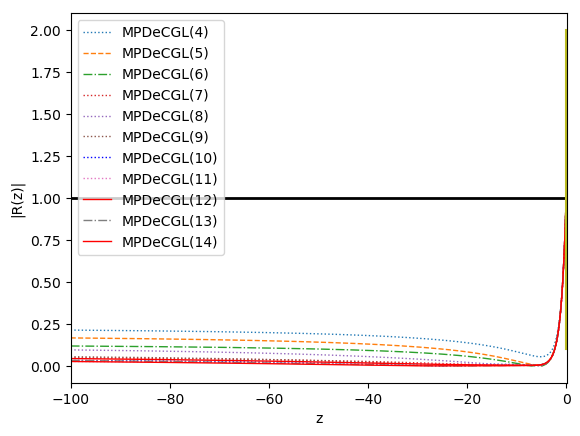} 
		\caption{Gauss--Lobtto points}
		\label{fig:GL_22}
	\end{subfigure}%
	~
	\begin{subfigure}[b]{0.49\textwidth}
		\includegraphics[width=\textwidth]{%
			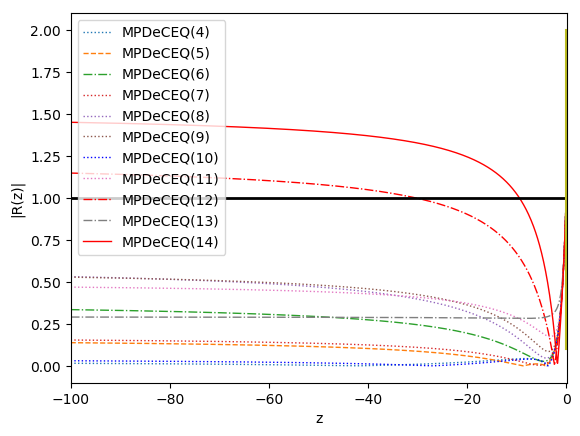} 
		\caption{Equispaced points in MPDeC}
		\label{fig:ESP_22}
	\end{subfigure}%
	\caption{Absolute value of the stability function over z}
	\label{fig:stability}
\end{figure}

\section{Numerical Experiments}\label{se_numerics}

In the following part, we verify the theoretical results of our considered modified Patankar schemes by numerical simulations. 
For the MPRK schemes we will consider more general systems where for the MPDeC schemes 
we restrict ourselves as described above to a $2\times 2$ system.

\subsection{Numerical Simulations of MPRK Schemes}

Due to our investigations in Section \ref{se_MPRK}, all the considered MPRK schemes are unconditionally stable theoretically
and our numerical experience will demonstrate this as well. 
For comparison, we will focus on the same tests as suggested and investigated in \cite{IKM2D,IKMSys,huang2022stability}. 
We consider the following three different settings:
\begin{enumerate}
	\item Test problem with exclusively real eigenvalues
	\item Test problem with complex eigenvalues 
	\item Test problem with double zero eigenvalues
\end{enumerate}

\subsection*{Test problem with exclusively real eigenvalues}
As the first test, we consider the linear initial  value problem (IVP)
\begin{equation}
	\by'=100\begin{pmatrix*}[r]
		-2 & 1 & 1\\
		1& -2& 1 \\
		1 & 3 &-2
	\end{pmatrix*} \by \text{ with } 
	\by(0)= \begin{pmatrix*}[r]
		1\\
		9\\
		5
	\end{pmatrix*}.
	\label{eq_numeris_real}
\end{equation}
The IVP \eqref{eq_numeris_real} has only positive off-diagonal elements inside the matrix and is therefore a Metzler matrix. Together with the positive initial values, this ensures that each component of the solution of the IVP is positive for all times. 
The eigenvalues are given by $\lambda_1=0, \lambda_2=-300$ and $\lambda_3=-500$ and the 
analytical solution is given by 
$$
\by(t)=c_1\begin{pmatrix*}[r]
	5\\
	3\\
	7
\end{pmatrix*}+c_2
\begin{pmatrix*}[r]
	-1\\
	0\\
	1
\end{pmatrix*}
\mathrm{e}^{-300t}
+c_3
\begin{pmatrix*}[r]
	0\\
	-1\\
	1
\end{pmatrix*}
\mathrm{e}^{-500t}
$$
with $(c_1,c_2,c_3)=(1,4,6)$. The steady state is given by $\by^*=(5,3,7)^T.$
The zero eigenvalue is simple and hence we obtain exactly one linear invariant given by $\mathbf{1}^T \by$. Therefore, the positive system is also conservative, e.g. $\sum_{i=1}^3 y_i^n=15$ for all $n\in \N_0$. To demonstrate the stability of the MPRK methods, we select three different types focusing on MPRK32, MPRK43(0.5) and MPRK43(0.9,0.6). 
In Figure \ref{fig:stability_real}, the analytical and numerical results are shown. 
In \ref{fig:short}, we plot it using to a comparably small time step of $\dt=0.05$, whereas
in \ref{fig:long} the time step is $\dt=25$. In both cases our approximated solution converges towards the fixed point even for a big time step. 
Note that we still have oscillations except for the MPRK32 scheme, we refer again to \cite{torlo2021issues} for a more detailed investigation on this topic. 

\begin{figure}[tb]
	\centering 
	\begin{subfigure}[b]{0.49\textwidth}
		\includegraphics[width=\textwidth]{%
			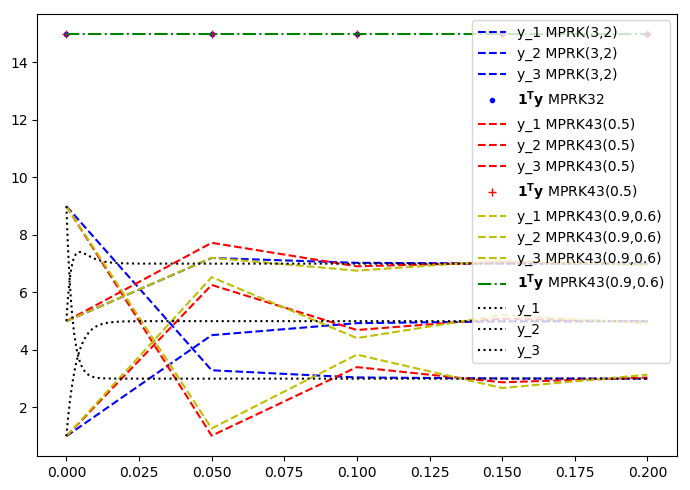} 
		\caption{$\dt=0.05$}
		\label{fig:short}
	\end{subfigure}%
	~
	\begin{subfigure}[b]{0.49\textwidth}
		\includegraphics[width=\textwidth]{%
			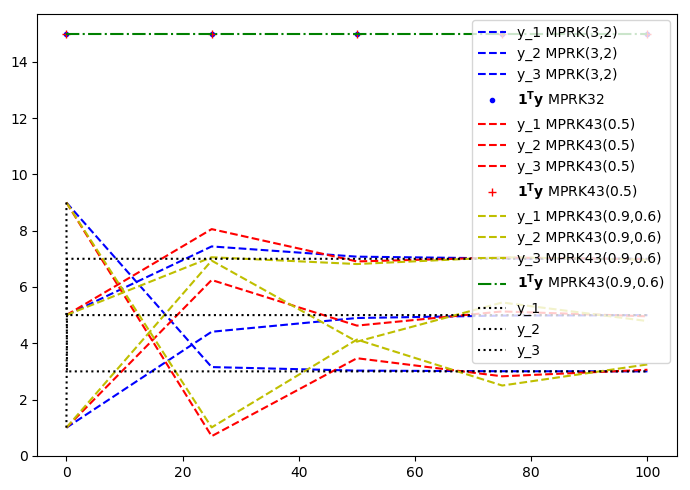} 
		\caption{$\dt=25$}
		\label{fig:long}
	\end{subfigure}%
	\caption{Approximation of \eqref{eq_numeris_real}}
	\label{fig:stability_real}
\end{figure}

\subsection*{Test problem with complex eigenvalues}

The second test is the IVP
\begin{equation}
	\by'=100\begin{pmatrix*}[r]
		-4 & 3 & 1\\
		2& -4& 3 \\
		2 & 1 &-4
	\end{pmatrix*} \by \text{ with } 
	\by(0)= \begin{pmatrix*}[r]
		9\\
		20\\
		8
	\end{pmatrix*}.
	\label{eq_numeris_complex}
\end{equation}
The system matrix in  \eqref{eq_numeris_complex}  is a  Metzler matrix
and we have the eigenvalues  $\lambda_1=0, \lambda_2=100(-6+\mathrm{i})$ and $\lambda_3=\overline{\lambda}$. The 
analytical solution is given by 
\begin{equation}
	\begin{aligned}
		\by(t)=\begin{pmatrix*}[r]
			13\\
			14\\
			10
		\end{pmatrix*}-&2\mathrm{e}^{-600t} \left( \cos(100t)
		\begin{pmatrix*}[r]
			-1\\
			0\\
			1
		\end{pmatrix*}
		-\sin (100t) 
		\begin{pmatrix*}[r]
			1\\
			-1\\
			9
		\end{pmatrix*}
		\right)\\
		-&6\mathrm{e}^{-600t} \left( \cos(100t)
		\begin{pmatrix*}[r]
			1\\
			-1\\
			0
		\end{pmatrix*}
		+\sin (100t) 
		\begin{pmatrix*}[r]
			-1\\
			0\\
			1
		\end{pmatrix*}
		\right)
	\end{aligned}\label{eq_solution}
\end{equation}
and the steady state is $\by^*= (13, 14,10)^T$. Also in this case a rapid decrease to $\by^*$ can be expected through the behaviour of \eqref{eq_solution}. 
In Figure \ref{fig:stability_complex}, we see again the numerical solution compared to the analytical one. All schemes are stable and converge towards the steady state solution. The conservation property is ensured  as well and can be recognized in \ref{fig:long_2}.

\begin{figure}[tb]
	\centering 
	\begin{subfigure}[b]{0.49\textwidth}
		\includegraphics[width=\textwidth]{%
			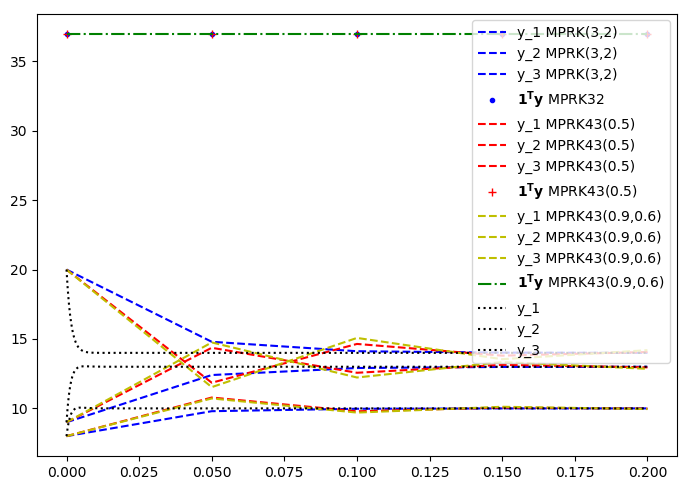} 
		\caption{$\dt=0.025$}
		\label{fig:short_1}
	\end{subfigure}%
	~
	\begin{subfigure}[b]{0.49\textwidth}
		\includegraphics[width=\textwidth]{%
			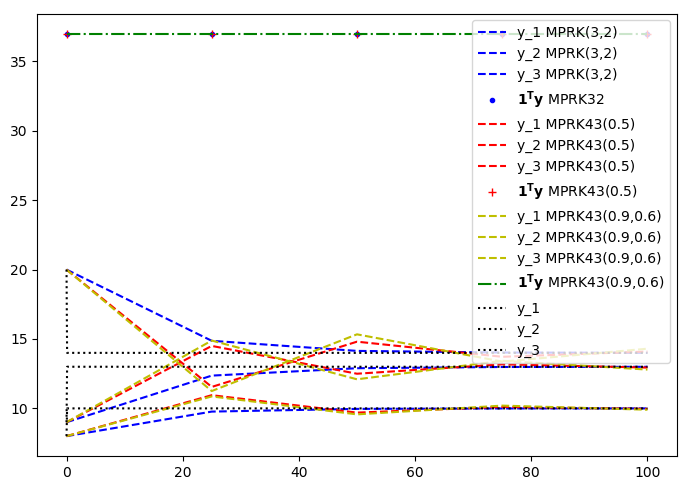} 
		\caption{$\dt=25$}
		\label{fig:long_2}
	\end{subfigure}%
	\caption{Approximation of \eqref{eq_numeris_complex}}
	\label{fig:stability_complex}
\end{figure}

\subsection*{Test problem with double zero eigenvalues}

The third test is the linear IVP
\begin{equation}
	\by'=100\begin{pmatrix*}[r]
		-2& 0 & 0& 1\\
		0& -4& 3 & 0 \\
		0 & 4 &-4& 0 \\
		2 & 0 & 0 &-1 
	\end{pmatrix*} \by \text{ with } 
	\by(0)= \begin{pmatrix*}[r]
		4\\
		1\\
		9 \\
		1
	\end{pmatrix*}.
	\label{eq_numeris_complexdouble}
\end{equation}
The system matrix in  \eqref{eq_numeris_complexdouble}  is a  Metzler matrix
with double zero eigenvalues  $\lambda_1= \lambda_2=0$.
Therefore, we obtain besides $\mathbf{1}^T \by$ (the conservative property), a second linear invariant $\mathbf{n}^T \by$ with $\mathbf{n}^T=(1,2,2,1)^T$.
Using the remaining eigenvalues $\lambda_3=-300$ and
$\lambda_4=-700$ and corresponding eigenvectors, the analytical solution is given by
\begin{equation*}
	\by(t)=c_1\begin{pmatrix*}[r]
		0\\
		1\\
		4/3\\
		0
	\end{pmatrix*}+c_2
	\begin{pmatrix*}[r]
		1\\
		0\\
		0\\
		2
	\end{pmatrix*}
	+
	c_3
	\mathrm{e}^{-700t}
	\begin{pmatrix*}[r]
		0 \\
		1\\
		-1\\
		0
	\end{pmatrix*}
	+c_4\mathrm{e}^{-300t}
	\begin{pmatrix*}[r]
		1\\
		0\\
		0\\1
	\end{pmatrix*}
\end{equation*}
with $(c_1,c_2,c_3,c_4)=(30/7,5/3, -23/7, 7/3)$.
The steady state will be reached quite fast and it is
given by $\by^*=\frac{1}{21}(7,90,120,70)^T$. In Figures \ref{fig:stability_2_short} and 
\ref{fig:stability_2_long}, we see the numerical results for our selected schemes for different $\dt$. All schemes converge to the steady state solutions and preserve both linear invariants as presented in the figures. 

\begin{figure}[tb]
	\centering 
	\begin{subfigure}[b]{0.33\textwidth}
		\includegraphics[width=\textwidth]{%
			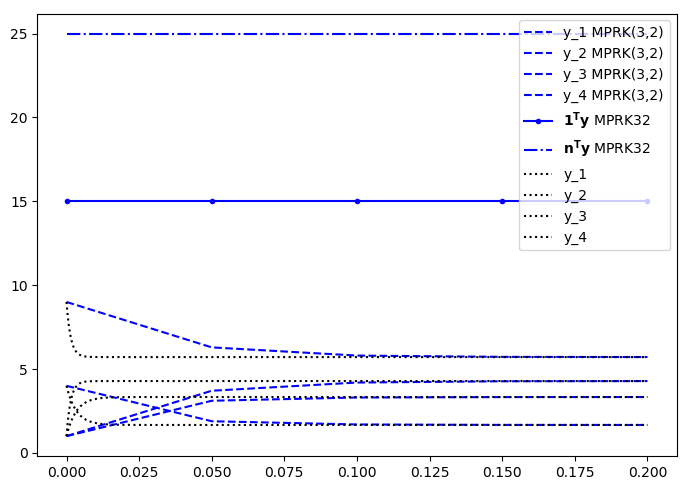} 
		\caption{MPRK32}
		\label{short_3_b_1}
	\end{subfigure}%
	~
	\begin{subfigure}[b]{0.33\textwidth}
		\includegraphics[width=\textwidth]{%
			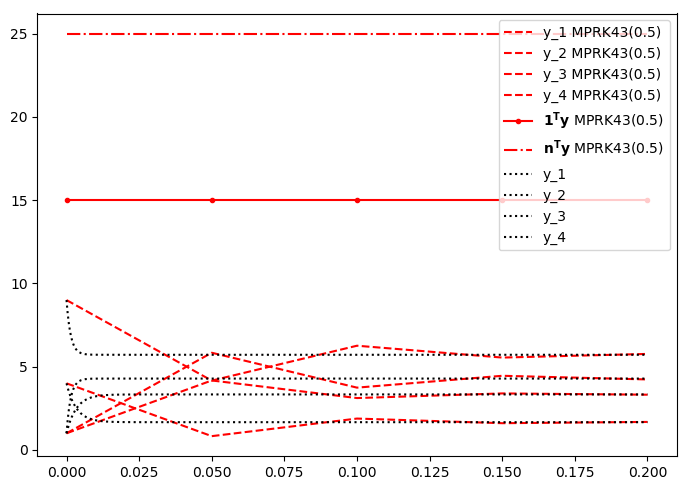} 
		\caption{MPRK43(0.5)}
		\label{short_3_bbb}
	\end{subfigure}%
	~
	\begin{subfigure}[b]{0.33\textwidth}
		\includegraphics[width=\textwidth]{%
			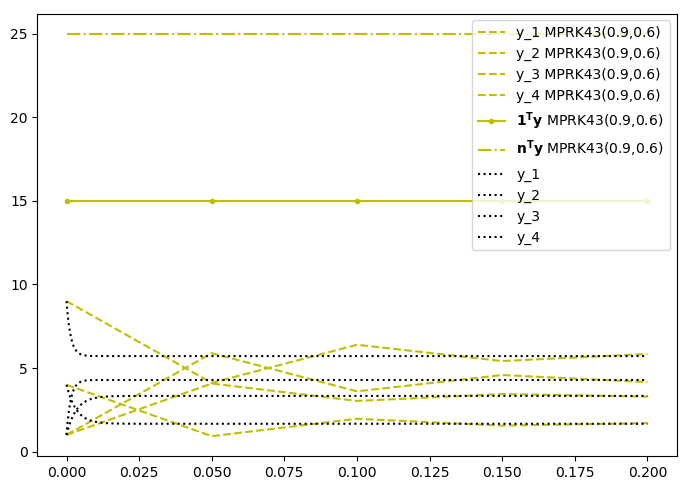} 
		\caption{MPRK43(0.9,0.6)}
		\label{short_3_b_3}
	\end{subfigure}%
	\caption{Approximation of \eqref{eq_numeris_complexdouble} with $\dt=0.05$}
	\label{fig:stability_2_short}
\end{figure}

\begin{figure}[tb]
	\centering 
	\begin{subfigure}[b]{0.33\textwidth}
		\includegraphics[width=\textwidth]{%
			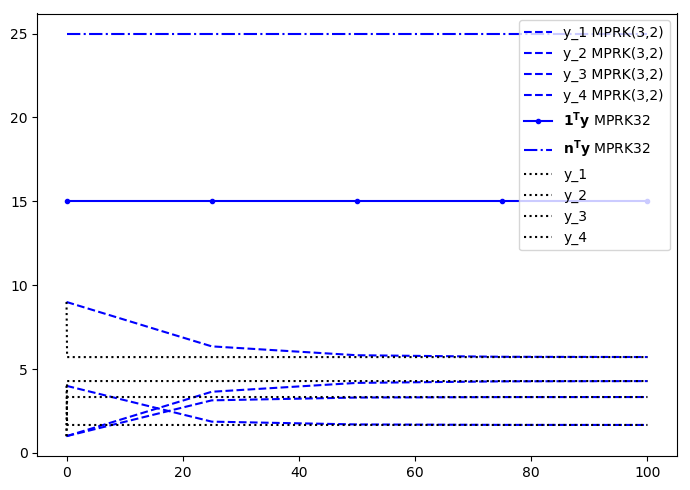} 
		\caption{MPRK32}
		\label{short_3_b_4}
	\end{subfigure}%
	~
	\begin{subfigure}[b]{0.33\textwidth}
		\includegraphics[width=\textwidth]{%
			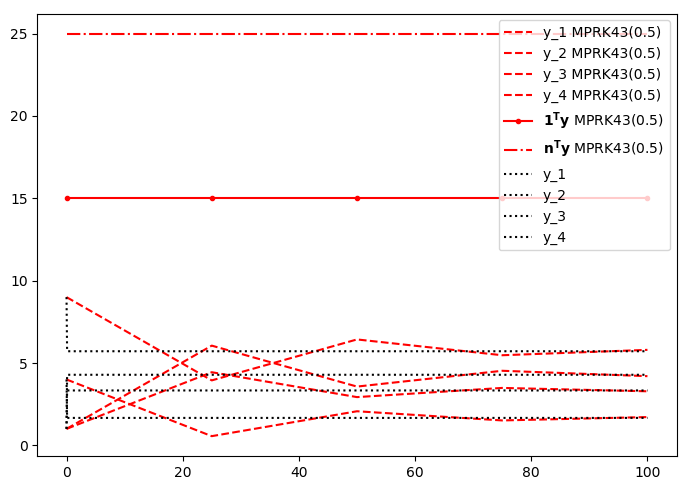} 
		\caption{MPRK43(0.5)}
		\label{short_3_b_5}
	\end{subfigure}%
	~
	\begin{subfigure}[b]{0.33\textwidth}
		\includegraphics[width=\textwidth]{%
			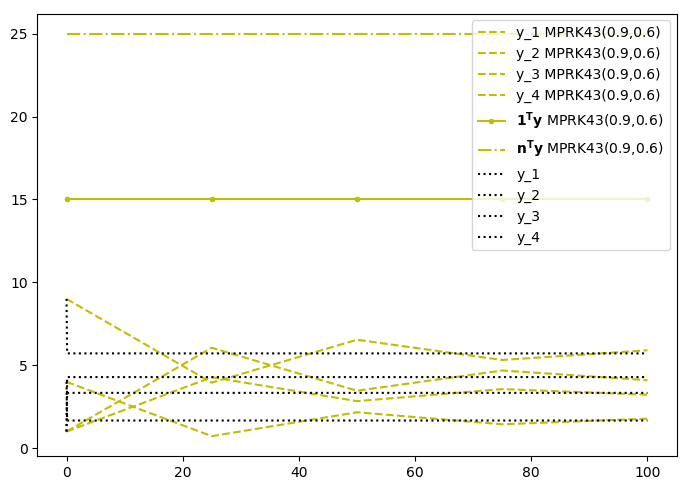} 
		\caption{MPRK43(0.9,0.6)}
		\label{short_3_b_6}
	\end{subfigure}%
	\caption{Approximation of \eqref{eq_numeris_complexdouble} with $\dt=25$}
	\label{fig:stability_2_long}
\end{figure}

\subsection{Numerical Simulations of MPDeC Schemes}
We consider the initial value problem 
\begin{equation}
	\by'=\begin{pmatrix*}[r]
		-25 & 25\\
		25  & -25 
	\end{pmatrix*} \by \text{ with } 
	\by(0)= \begin{pmatrix*}[r]
		0.998\\
		0.002
	\end{pmatrix*}.
	\label{eq_numeris_2_2}
\end{equation}
The nonzero eigenvalue $\lambda=-50$ and the analytical solutions is given by 
$$
\by(t)=\frac{1}{2}\begin{pmatrix*}[r]
	1\\
	1
\end{pmatrix*}+0.498 
\begin{pmatrix*}[r]
	1\\
	-1
\end{pmatrix*}
\mathrm{e}^{-50t}.
$$
The steady state is given by $\by^*=(0.5,0.5)^T.$
In Figure \ref{fig:stability_dec}, the numerical and analytical solutions are plotted for several different MPDeC schemes. 

\begin{figure}[tb]
	\centering 
	\begin{subfigure}[b]{0.49\textwidth}
		\includegraphics[width=\textwidth]{%
			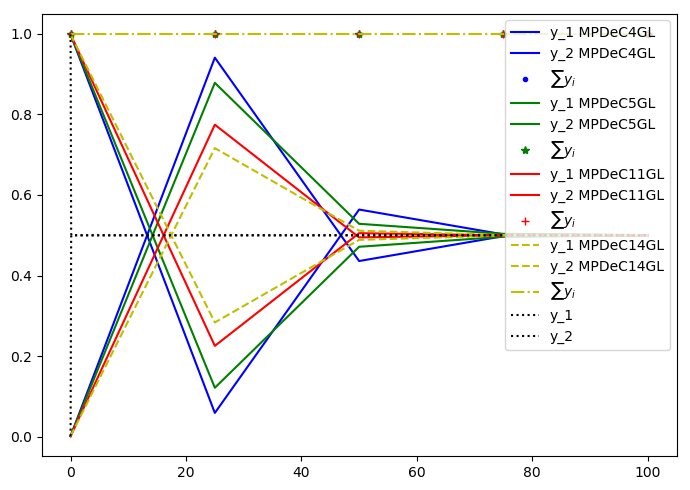} 
		\caption{Gauss--Lobtto points}
		\label{fig:GL_2}
	\end{subfigure}%
	~
	\begin{subfigure}[b]{0.49\textwidth}
		\includegraphics[width=\textwidth]{%
			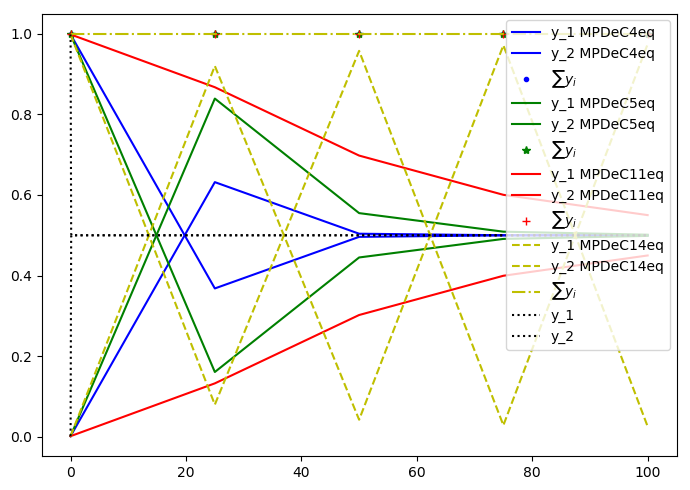} 
		\caption{Equispaced points in MPDeC}
		\label{fig:ESP_2}
	\end{subfigure}%
	\caption{Approximation with $\dt=25$ }
	\label{fig:stability_dec}
\end{figure}

From our investigation in Section \ref{se_dec}, 
we know that using Gauss--Lobatto points leads always to stable and conservative approximations as can be also recognized in \ref{fig:GL_2}, however, for equidistant points, we obtain problems inside the approximation for higher-order approximations. This can also be seen in our numerical approximation for the IVP \eqref{eq_numeris_2_2} in Figure \ref{fig:ESP_2}. As we mentioned earlier this behaviour is not surprising since already the classical DeC method (without using the modified Patankar trick) has problems for high-order approximations, i.\,e.\  for $8$-th order or higher, if equidistant points are used. Obviously, MPDeC(14) is unstable and alternates around the steady state solution. Nevertheless, we can obtain a stable discretization for this test case if we decrease the time step. We can estimate the time step using the eigenvalue $\lambda=50$ and the value $z^* \approx -9.403$ satisfying $\lvert R(z^*)\rvert=1$ in Figure \ref{fig:ESP}. From $z=\dt\lambda$ we thus find the bound $\dt \leq 0.188$. 
In Figure \ref{fig:stability_dec_2}, we see an unstable behaviour for $\dt =0.2$ whereas for $\dt\approx 0.17$ also MPDeC(14) converges to the steady state solution. Indeed, increasing $\dt$ up to $0.188$ to obtain stable approximations, however, the convergence rate will also be quite slow and we need more time to reach the steady state. 
\begin{figure}[tb]
	\centering 
	\begin{subfigure}[b]{0.49\textwidth}
		\includegraphics[width=\textwidth]{%
			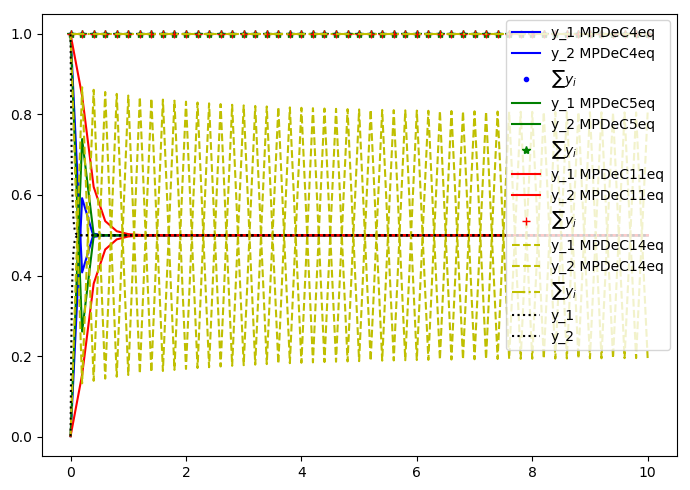} 
		\caption{$\dt=0.2$}
		\label{fig:ESP_Short}
	\end{subfigure}%
	~
	\begin{subfigure}[b]{0.49\textwidth}
		\includegraphics[width=\textwidth]{%
			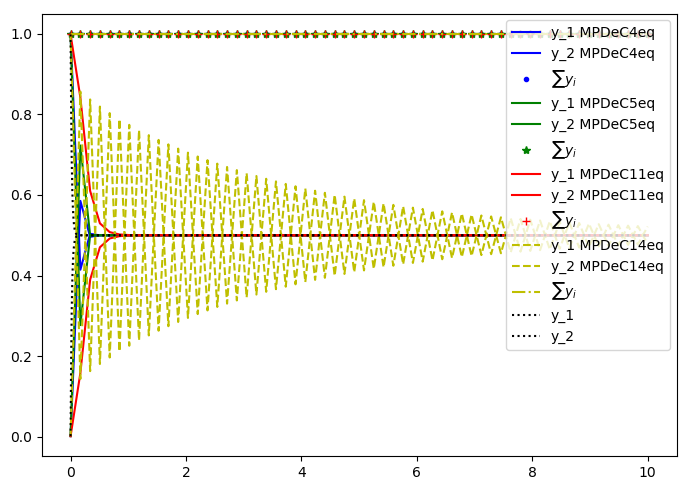} 
		\caption{$\dt\approx0.17$}
		\label{fig:ESP_short}
	\end{subfigure}%
	\caption{Numerical approximation of \eqref{eq_numeris_2_2} }
	\label{fig:stability_dec_2}
\end{figure}  
Note that the presented theory only claims a local convergence of the iterates towards the steady state. 
This can be seen, for instance, in the last example from \cite[Figure B.9]{torlo2021issues}, where it is shown that the stable MPDeC(8) scheme does not convergence to the correct steady-state. 
Nevertheless, if we violate the stability condition, we can start arbitrary close to the steady state solution, and still, the iterates will not converge to $\by^*$. 
To clarify this, we change our IVP \eqref{eq_numeris_2_2} to the following condition
\begin{equation}
	\by'=\begin{pmatrix*}[r]
		-0.5 & 0.5\\
		0.5  & -0.5
	\end{pmatrix*} \by \text{ with } 
	\by(0)=\by^*+ 10^{-6}\begin{pmatrix*}[r]
		1\\
		-1
	\end{pmatrix*}.
	\label{eq_numeris_2_3}
\end{equation}
In Figure \ref{fig:stability_dec_3}, we plot the numerical solution  $y_2$ calculated with order MPDeC(14) using equispaced points. The time step is selected with respect to the eigenvalue $\lambda_1=-1$, so that $\dt= \frac{z}{\lambda}=\lvert z\rvert$ holds. In \ref{convergence_2}, we see a slow convergence against the steady state. In Figures \ref{convergence_3}-\ref{convergence_4}, the conditions for convergence are not fulfilled and we see this also in our figures. A similar behaviour can be seen for $y_1$ analogously.  

\begin{figure}[tb]
	\centering 
	\begin{subfigure}[b]{0.24\textwidth}
		\includegraphics[width=\textwidth]{%
			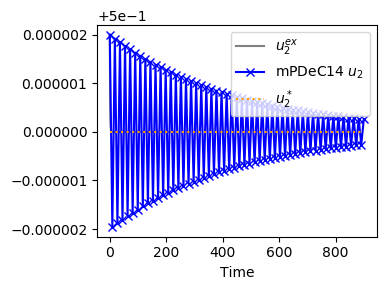} 
		\caption{$z=-9$}
		\label{convergence_1}
	\end{subfigure}%
	~
	\begin{subfigure}[b]{0.24\textwidth}
		\includegraphics[width=\textwidth]{%
			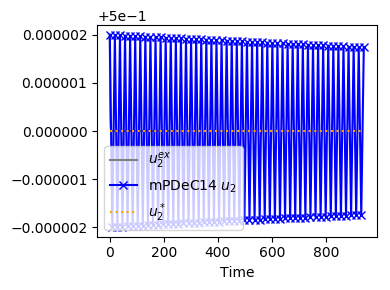} 
		\caption{$z=z_{c}$, $\lvert R(z_{c})\rvert<1$}
		\label{convergence_2}
	\end{subfigure}%
	~
	\begin{subfigure}[b]{0.24\textwidth}
		\includegraphics[width=\textwidth]{%
			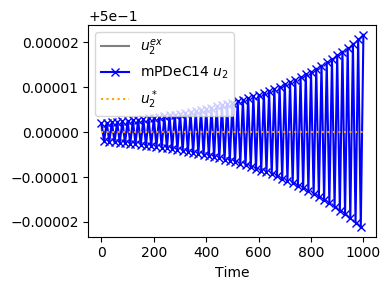} 
		\caption{$z=-10$}
		\label{convergence_3}
	\end{subfigure}%
	~
	\begin{subfigure}[b]{0.24\textwidth}
		\includegraphics[width=\textwidth]{%
			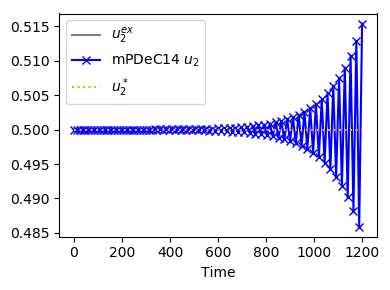} 
		\caption{$z=-12$}
		\label{convergence_4}
	\end{subfigure}%
	\caption{Numerical approximation of $y_2$ of \eqref{eq_numeris_2_3} }
	\label{fig:stability_dec_3}
\end{figure}  
\section{Conclusion}\label{se_conclusion}

In this paper, we have investigated the stability conditions for several modified Patankar schemes. We were able to demonstrate that the MPRK(3,2) scheme proposed in \cite{torlo2021issues} and the MPRK43($\gamma$) schemes from \cite{kopecz2018unconditionally} are stable considering  general positive linear $N\times N$ systems. Further, we  provided the stability functions for the MPRK43($\alpha, \beta)$ schemes and investigated it for $N\leq 4$. 
For a fixed method, the conditions \eqref{eq:assumption.prop.MPRK43(alpha,beta)} can easily be checked whereas through a numerical study we have concluded that these conditions should be as well fulfilled for the general  parameter selection. However, a formal proof is still missing. This remaining investigation will also be part of future research. 
We were  also able to compute the stability function for MPDeC schemes and investigated the schemes numerically. We have seen that using Gauss--Lobatto points, we always obtain a stable scheme whereas for equispaced points stability issues have been recognized for higher-order methods and large $\dt$. This fact is not surprising since already DeC methods show stability issues  if equispaced points are used and also the investigation  in \cite{torlo2021issues} confirms those problems. We have seen that the $\dt$ bound can be estimated from stability functions and eigenvalues.

In the future, we plan to continue our investigation in several directions. First, we aim to investigate MPRK43($\alpha,\beta$) and MPDeC in more general setups. Furthermore, we will consider non-conservative Patankar systems as the ones proposed in \cite{chertock2015steady}. Later, we combine MP schemes with proper selected space discretization and consider hyperbolic conservation/balance laws as it was already successfully done in \cite{ciallella2021arbitrary}.
 However, it is not clear what the stability and convergence properties of such combined schemes are.

\section{Appendix}\label{se_Appendix}
We perform a numerical validation that Assumptions \eqref{eq:assumption.prop.MPRK43(alpha,beta)} are always fulfilled for the general parameter selection. 
\begin{figure}[tb]
	\centering 
	\begin{subfigure}[b]{0.33\textwidth}
		\includegraphics[width=\textwidth]{%
			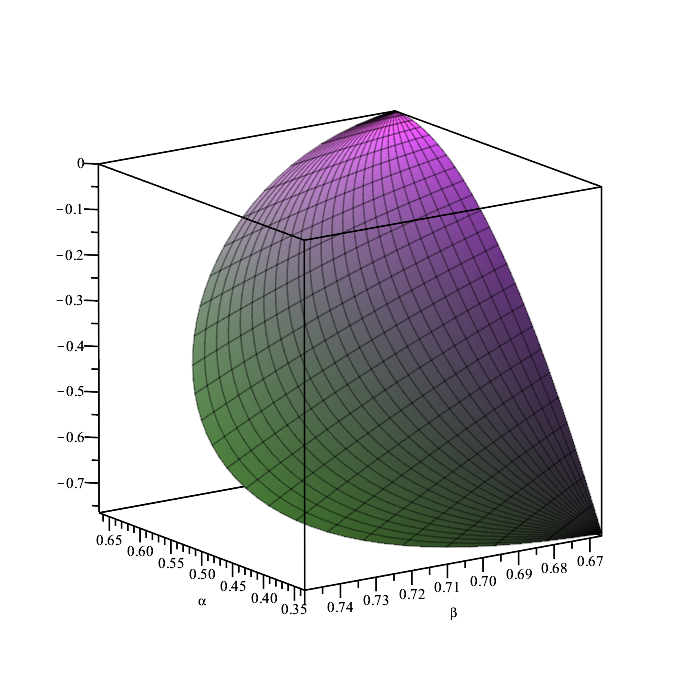} 
		\caption{$\frac13\leq \alpha\leq\frac23\leq \beta\leq 3\alpha(1-\alpha)$}
		\label{short_3_b_8}
	\end{subfigure}%
	~
	\begin{subfigure}[b]{0.33\textwidth}
		\includegraphics[width=\textwidth]{%
			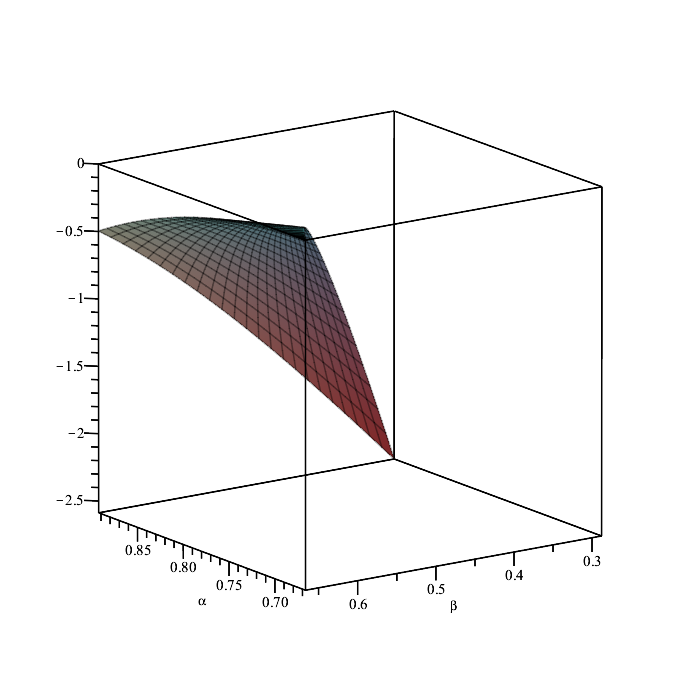} 
		\caption{$3\alpha(1-\alpha)\leq \beta \leq \frac23\leq \alpha<\alpha_0$}
		\label{short_3_b_9}
	\end{subfigure}%
	~
	\begin{subfigure}[b]{0.33\textwidth}
		\includegraphics[width=\textwidth]{%
			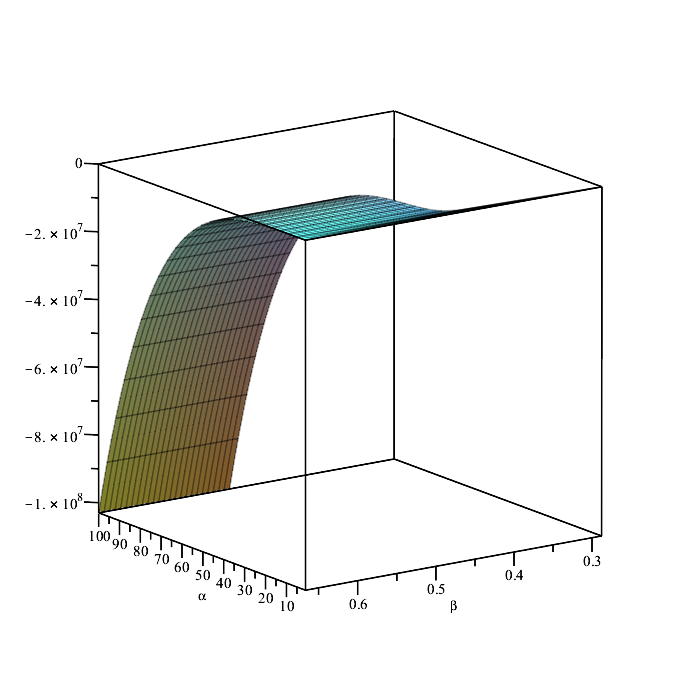} 
		\caption{$\tfrac{3\alpha-2}{6\alpha-3}\leq \beta \leq \frac23$, $\alpha>\alpha_0$}
		\label{short_3_b_10}
	\end{subfigure}%
	\caption{Plot of $(-a_1a_2+b_1b_2+a_3-b_3)(\beta-\alpha)^2$ from \eqref{eq:assumption.prop.MPRK43(alpha,beta)}.}
	\label{fig:c2}
\end{figure}

\begin{figure}[tb]
	\centering 
	\begin{subfigure}[b]{0.33\textwidth}
		\includegraphics[width=\textwidth]{%
			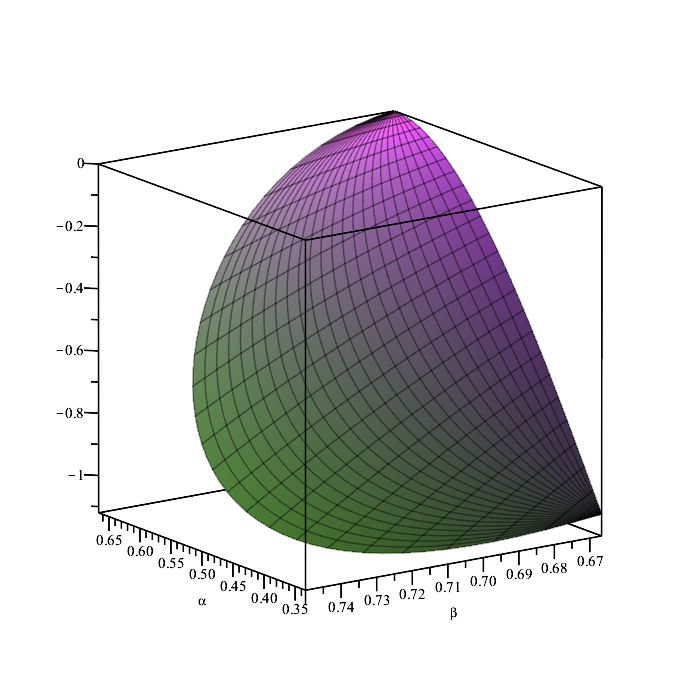} 
		\caption{$\frac13\leq \alpha\leq\frac23\leq \beta\leq 3\alpha(1-\alpha)$}
		\label{short_3_b_11}
	\end{subfigure}%
	~
	\begin{subfigure}[b]{0.33\textwidth}
		\includegraphics[width=\textwidth]{%
			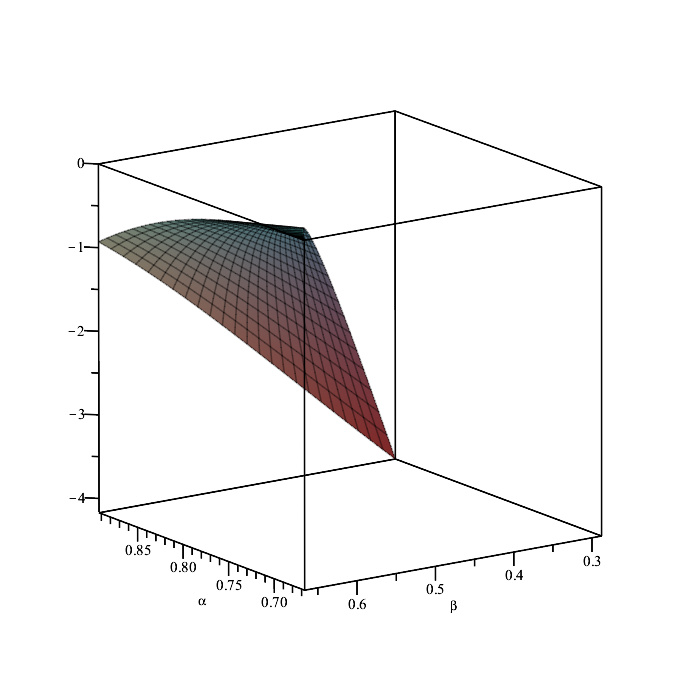} 
		\caption{$3\alpha(1-\alpha)\leq \beta \leq \frac23\leq \alpha<\alpha_0$}
		\label{short_3_b_12}
	\end{subfigure}%
	~
	\begin{subfigure}[b]{0.33\textwidth}
		\includegraphics[width=\textwidth]{%
			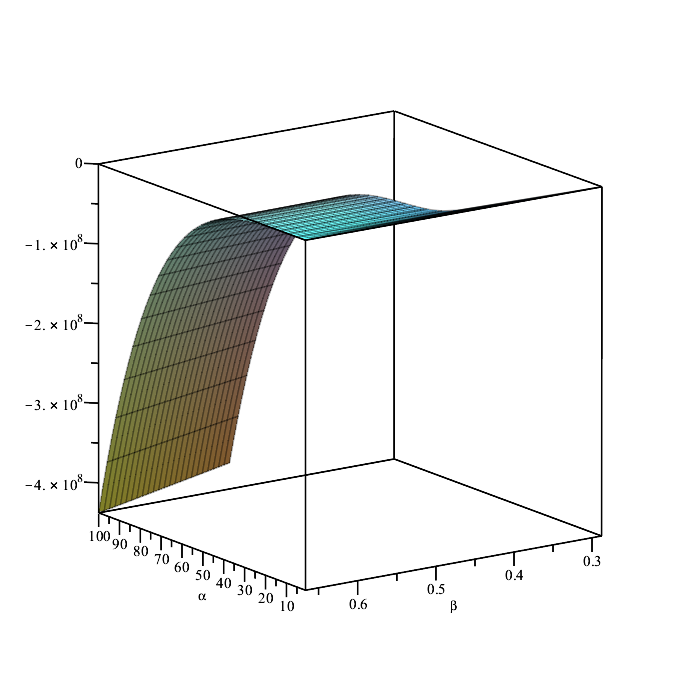} 
		\caption{$\tfrac{3\alpha-2}{6\alpha-3}\leq \beta \leq \frac23$, $\alpha>\alpha_0$}
		\label{short_3_b_13}
	\end{subfigure}%
	\caption{Plot of $(a_2^2-b_2^2-2a_4+2b_4)(\beta-\alpha)^2$ from \eqref{eq:assumption.prop.MPRK43(alpha,beta)}.}
	\label{fig:c3}
\end{figure}

\begin{figure}[tb]
	\centering 
	\begin{subfigure}[b]{0.33\textwidth}
		\includegraphics[width=\textwidth]{%
			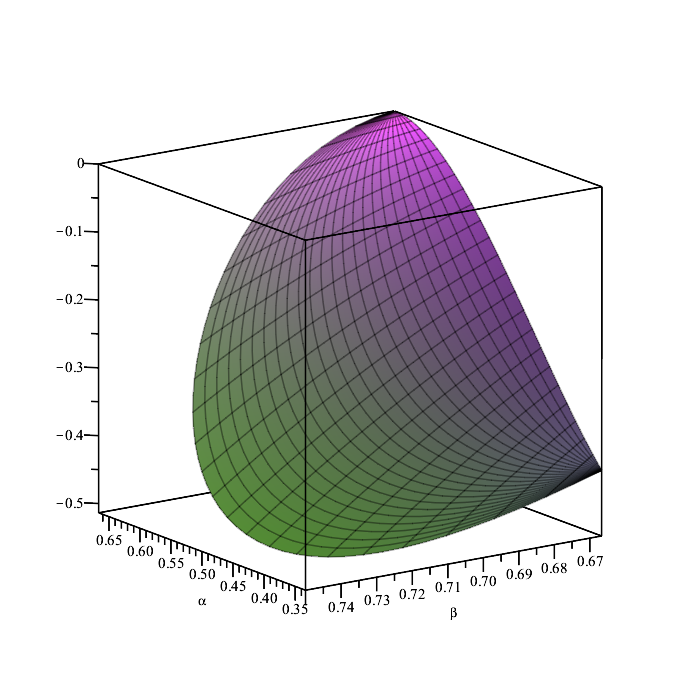} 
		\caption{$\frac13\leq \alpha\leq\frac23\leq \beta\leq 3\alpha(1-\alpha)$}
		\label{short_3_b_15}
	\end{subfigure}%
	~
	\begin{subfigure}[b]{0.33\textwidth}
		\includegraphics[width=\textwidth]{%
			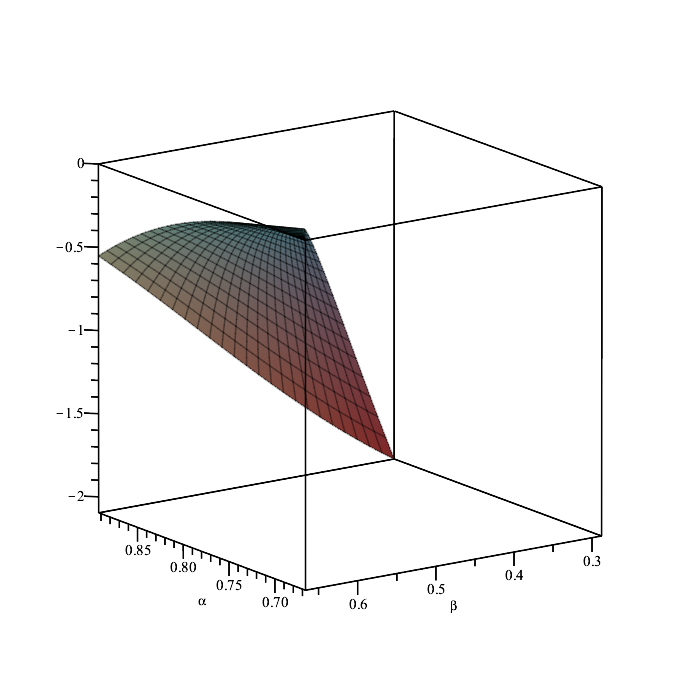} 
		\caption{$3\alpha(1-\alpha)\leq \beta \leq \frac23\leq \alpha<\alpha_0$}
		\label{short_3_b_16}
	\end{subfigure}%
	~
	\begin{subfigure}[b]{0.33\textwidth}
		\includegraphics[width=\textwidth]{%
			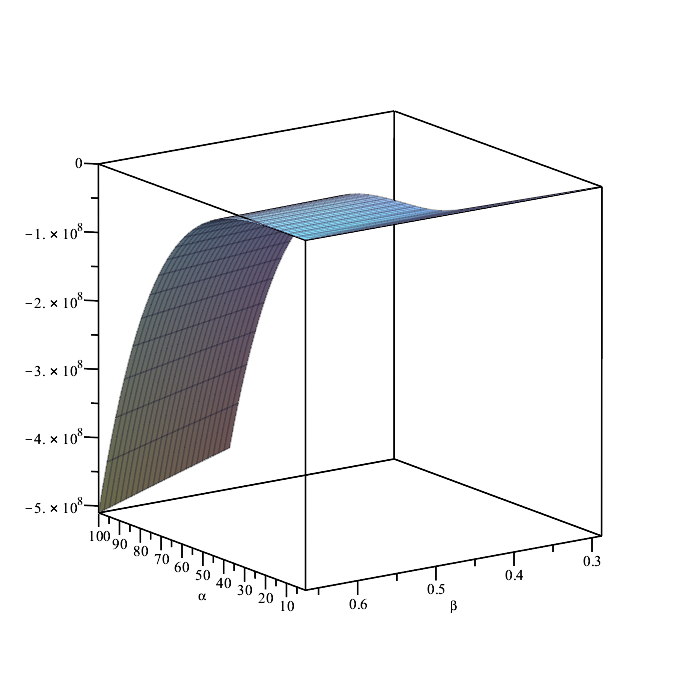} 
		\caption{$\tfrac{3\alpha-2}{6\alpha-3}\leq \beta \leq \frac23$, $\alpha>\alpha_0$}
		\label{short_3_b_17}
	\end{subfigure}%
	\caption{Plot of $(a_1a_4- a_2a_3-b_1b_4+b_2b_3)(\beta-\alpha)^2$ from \eqref{eq:assumption.prop.MPRK43(alpha,beta)}.}
	\label{fig:c4}
\end{figure}

\begin{figure}[tb]
	\centering 
	\begin{subfigure}[b]{0.33\textwidth}
		\includegraphics[width=\textwidth]{%
			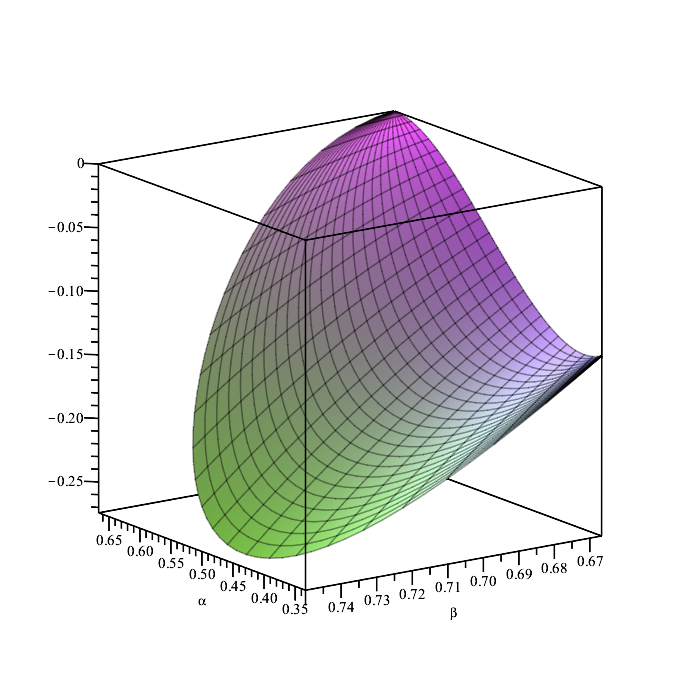} 
		\caption{$\frac13\leq \alpha\leq\frac23\leq \beta\leq 3\alpha(1-\alpha)$}
		\label{short_3_b_18}
	\end{subfigure}%
	~
	\begin{subfigure}[b]{0.33\textwidth}
		\includegraphics[width=\textwidth]{%
			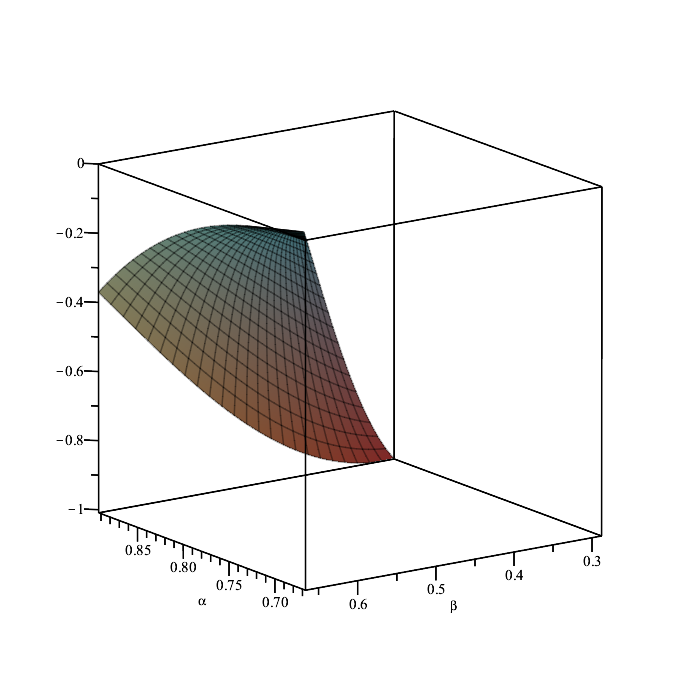} 
		\caption{$3\alpha(1-\alpha)\leq \beta \leq \frac23\leq \alpha<\alpha_0$}
		\label{short_3_b_19}
	\end{subfigure}%
	~
	\begin{subfigure}[b]{0.33\textwidth}
		\includegraphics[width=\textwidth]{%
			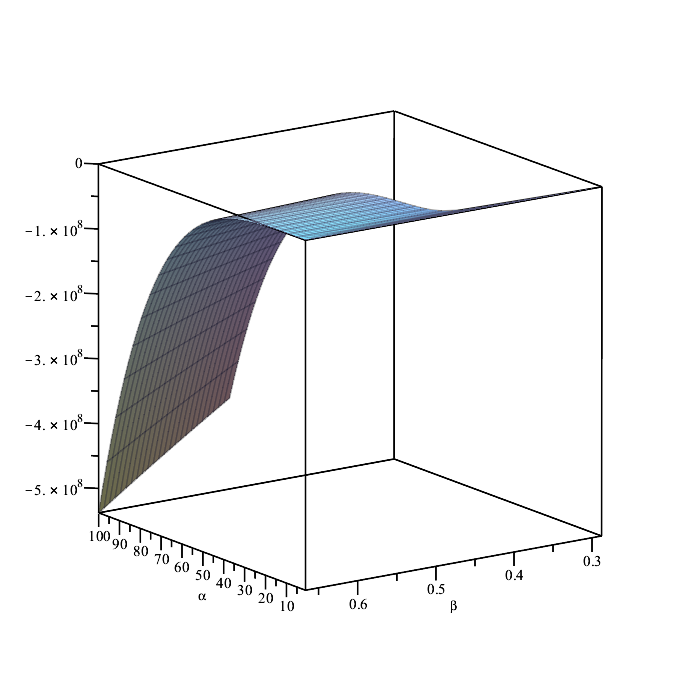} 
		\caption{$\tfrac{3\alpha-2}{6\alpha-3}\leq \beta \leq \frac23$, $\alpha>\alpha_0$}
		\label{short_3_b_20}
	\end{subfigure}%
	\caption{Plot of $(a_3^2-b_3^2)(\beta-\alpha)^2$ from \eqref{eq:assumption.prop.MPRK43(alpha,beta)}.}
	\label{fig:c5}
\end{figure}

\begin{figure}[tb]
	\centering 
	\begin{subfigure}[b]{0.33\textwidth}
		\includegraphics[width=\textwidth]{%
			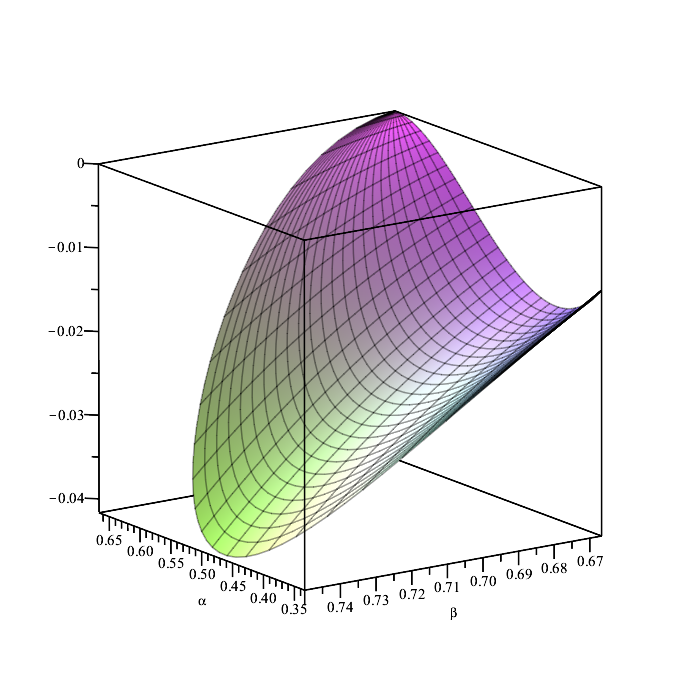} 
		\caption{$\frac13\leq \alpha\leq\frac23\leq \beta\leq 3\alpha(1-\alpha)$}
		\label{short_3_b_21}
	\end{subfigure}%
	~
	\begin{subfigure}[b]{0.33\textwidth}
		\includegraphics[width=\textwidth]{%
			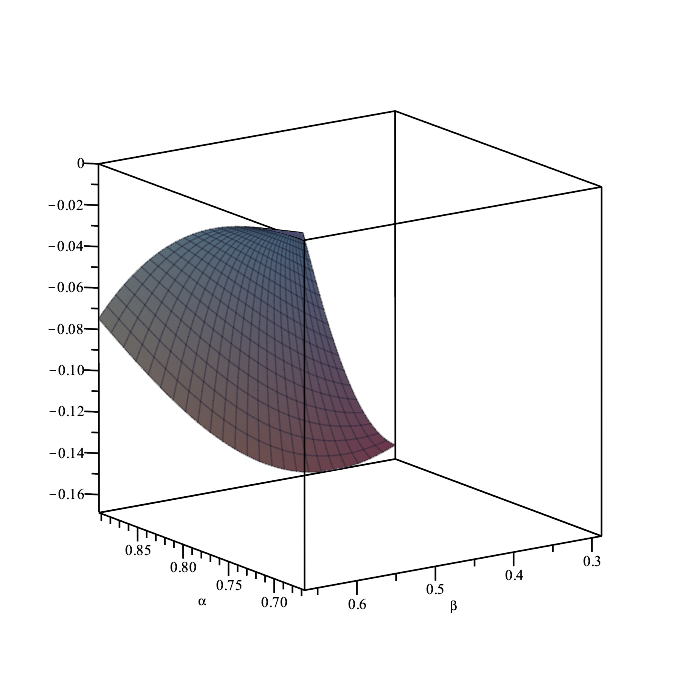} 
		\caption{$3\alpha(1-\alpha)\leq \beta \leq \frac23\leq \alpha<\alpha_0$}
		\label{short_3_b_22}
	\end{subfigure}%
	~
	\begin{subfigure}[b]{0.33\textwidth}
		\includegraphics[width=\textwidth]{%
			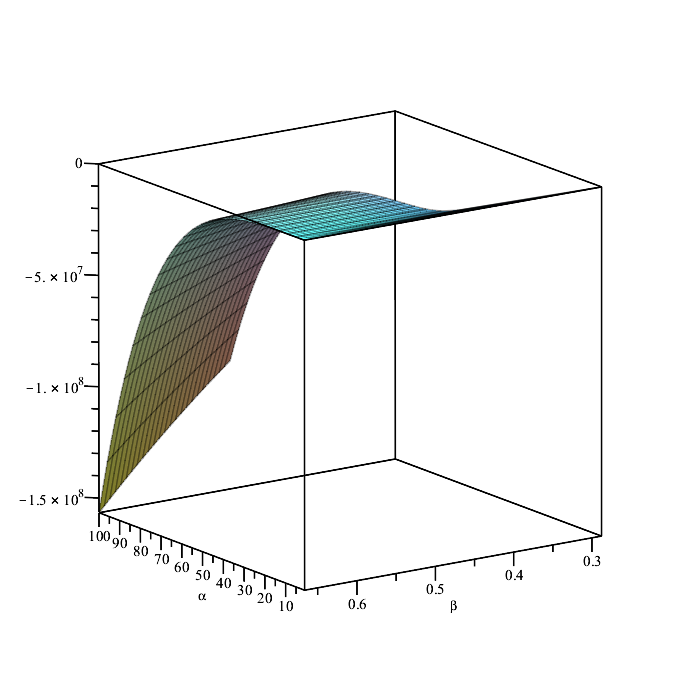} 
		\caption{$\tfrac{3\alpha-2}{6\alpha-3}\leq \beta \leq \frac23$, $\alpha>\alpha_0$}
		\label{short_3_b_23}
	\end{subfigure}%
	\caption{Plot of $(b_3b_4-a_3a_4)(\beta-\alpha)^2$ from \eqref{eq:assumption.prop.MPRK43(alpha,beta)}.}
	\label{fig:c6}
\end{figure}